\documentclass{article}

 \usepackage[nonatbib, preprint]{neurips_2019}

\usepackage[utf8]{inputenc} 
\usepackage[T1]{fontenc}    
\usepackage{hyperref}       
\usepackage{url}            
\usepackage{booktabs}       
\usepackage{amsfonts}       
\usepackage{nicefrac}       
\usepackage{microtype}      

\usepackage{multirow}

\usepackage{amsmath,amsfonts,amssymb,amsthm,array}
\usepackage{amsmath}
\usepackage{algorithm}
\usepackage{caption}
\usepackage{subcaption}
\usepackage{algpseudocode}
\usepackage{bm}
\usepackage{color}
\usepackage{graphicx}

\usepackage{longtable}

\newcommand{\compactify}{} 

\newcommand{\R}{\mathbb{R}}

\newcommand{\eqdef}{\triangleq}

\newcommand{\cA}{{\cal A}}
\newcommand{\cB}{{\cal B}}

\newcommand{\cL}{{\cal L}}

\newcommand{\cO}{{\cal O}}
\newcommand{\cP}{{\cal P}}

\newcommand{\cX}{{\cal X}}

\newcommand{\cW}{{\cal W}}

\newcommand{\mA}{{\bf A}}

\newcommand{\mI}{{\bf I}}

\newcommand{\norm}[1]{\left\| #1 \right\|}

\usepackage{tcolorbox}
\usepackage{pifont}
\definecolor{mydarkgreen}{RGB}{39,130,67}
\definecolor{mydarkred}{RGB}{192,47,25}
\newcommand{\green}{\color{mydarkgreen}}
\newcommand{\red}{\color{mydarkred}}
\newcommand{\cmark}{\green\ding{51}}%
\newcommand{\xmark}{\red\ding{55}}%

\theoremstyle{plain}
\newtheorem{theorem}{Theorem}[section]

\newtheorem{assumption}{Assumption}[section]

\newtheorem{lemma}[theorem]{Lemma}

\newtheorem{definition}[theorem]{Definition}

\newtheorem{corollary}[theorem]{Corollary}
\theoremstyle{remark}

\usepackage[colorinlistoftodos,bordercolor=orange,backgroundcolor=orange!20,linecolor=orange,textsize=scriptsize]{todonotes}

\title{MISO is Making a Comeback\\ With  Better Proofs and Rates}

\author{%
  Xun Qian\\
 KAUST\thanks{King Abdullah University of Science and Technology, Thuwal, Saudi Arabia}\\\
  \texttt{xun.qian@kaust.edu.sa} \\
  \And
   Alibek Sailanbayev \\
   KAUST\\\
   \texttt{alibek.sailanbayev@kaust.edu.sa} \\
   \AND
   Konstantin Mishchenko \\
   KAUST\\
   \texttt{konstantin.mishchenko@kaust.edu.sa} \\
   \And
   Peter Richt\'{a}rik \\
   KAUST and MIPT\thanks{Moscow Institute of Physics and Technology, Dolgoprudny, Russia.}\\
   \texttt{peter.richtarik@kaust.edu.sa} \\
}

\begin{document}

\maketitle

\begin{abstract}
MISO~\cite{mairal2013optimization}, also known as Finito~\cite{defazio2014finito}, was one of the first stochastic variance reduced methods discovered, yet its popularity is fairly low. Its initial analysis was significantly limited by the so-called Big Data assumption. Although the assumption was lifted in subsequent work using negative momentum, this introduced a new parameter and required knowledge of strong convexity and smoothness constants, which is rarely possible in practice. We rehabilitate the method by introducing a new variant that needs only smoothness constant and does not have any extra parameters. Furthermore, when removing the strong convexity constant from the stepsize, we present a new analysis of the method, which no longer uses the assumption that every component is strongly convex. This allows us to also obtain so far unknown nonconvex convergence of MISO. To make the proposed method efficient in practice, we derive minibatching bounds with arbitrary uniform sampling that lead to linear speedup when the expected minibatch size is in a certain range. Our numerical experiments show that MISO is a serious competitor to SAGA and SVRG and sometimes outperforms them on real datasets.
\end{abstract}

\section{Introduction}

We study smooth finite-sum problem
\begin{equation}\label{primal}
\compactify \min \limits_{x\in \mathbb{R}^d} f(x) \eqdef \frac{1}{n}  \sum\limits_{i=1}^n f_i(x),
\end{equation}
where each summand $f_i:\R^d\to \R$ has Lipschitz continuous gradient. This simple formulation is ubiquitous in many areas and in particular in machine learning. Minimizing empirical loss is typically done using stochastic first-order information and variance reduction is a notable technique that sometimes makes convergence orders of magnitude faster. 

Stochastic updates are particularly attractive when the given problem is minimization of loss over training data. In this case, minibatch variance reduced methods provide scalability and ease of implementation. A handful of such methods exist by now, and it seems that they all fall in one of a few categories and the difference within each category is  in relatively minor tweaks only. For instance, there is a number of algorithms that differ from the stochastic variance reduced method (SVRG)~\cite{SVRG} only in the loop length~\cite{S2GD, L-SVRG}, inexact updates~\cite{lei2016less, lei2017non} or in using a different output of its loop~\cite{shang2018vr}. Similarly, there are multiple variants of Katyusha~\cite{L-SVRG, zhou2018simple}, SAGA~\cite{hofmann2015variance, raj2018svrg} and SDCA~\cite{shalev2015sdca, shalev2016sdca}. While some of these works propose valuable contribution for specific applications, in terms of ideas we see little novelty. In fact, it is now understood that there is almost no conceptual difference between SVRG and SAGA and virtually the same proof can be used for both of them~\cite{hofmann2015variance, L-SVRG}. 

Following these observations, we categorize methods as follows, ignoring acceleration and other modifications. Those that are related to SVRG, SARAH and SAGA use memory to directly approximate full gradient. SDCA, Point-SAGA and other proximal methods essentially do something similar but with an implicit type of updates. And, finally, majorization-minimization (MM) optimization algorithms were studied extensively in the past in various communities~\cite{Lange-MM-book}, and in the context of variance reduction this was also a topic of active research a few years ago~\cite{mairal2013optimization, defazio2014finito}.

{\bf MM methods.} A distinct feature of the MM approach is its universality. Unlike the unbiasedness and decreasing variance of SVRG or SAGA, MM is a principle that goes far beyond optimization. Probably for this reason variants of MISO have been rediscovered by several groups of authors~\cite{mishchenko2018delay, mishchenko2018distributed, mokhtari2018surpassing, mokhtari2018iqn}, not to mention that MISO itself was simultaneously discovered in two works~\cite{mairal2013optimization, defazio2014finito}. We believe that universality of MM provides a good start for designing new stochastic optimization methods and is going to lead to more future discoveries. For instance, one can go beyond first-order minimization and use Newton-like updates by minimizing a surrogate function with Hessian information. In fact, there already exists a BFGS analogue of cyclic MISO that was motivated by similar ideas~\cite{mokhtari2018iqn}. Likewise, there has been designed a generalization of MISO that additionally works with $m$ proximable functions~\cite{ryu2017proximal}, but unfortunately the method uses stepsizes of order $\nicefrac{1}{L}$ rather than $\nicefrac{n}{L}$.

Despite the universality of the MM approach employed by MISO,  several of its characteristics  lower its potential. The current analysis is based either on minimizing upper~\cite{mairal2013optimization} or lower bounds~\cite{mairal2013optimization, bietti2017stochastic, lin2015universal, defazio2014finito}. The upper bound minimization principle is slow and, thus, of little interest, while the lower bound approach suffers from the  need to know the strong convexity constant, which is rarely available, and when it can be estimated, the estimates are not sharp. Furthermore, in modern applications it is common to parallelize computation leading to the question of scalability. Moreover, lower bounds are simply not valid when the objective function is not convex, which limits the method even more. 

{\bf Notation.} A table summarizing key notation can be found in Appendix~\ref{sec:notation_appendix} (Table~\ref{tbl:notation}).

\section{Assumptions and Contributions} \label{sec:contrib}

\subsection{Assumptions}

In the rest of the paper  we assume  that $f$ is bounded below, and that problem~\ref{primal} has a nonempty set of minimizers $\cX^*$.  Our results will rely on a subset of the following smoothness and convexity assumptions.

\begin{assumption}[Smoothness] \label{as:smooth} The functions $f_i$ are $L$-smooth, i.e.,  their gradients are Lipschitz with constant $L \geq 0$. Further,  $f$ is $L_f$-smooth. (Note that,  clearly, $L_f\leq L$.) 
\end{assumption}

\begin{assumption}[Convexity] \label{as:convex}  The function $f$ is $\mu$-strongly convex, where $\mu\geq 0$. So, $\mu=0$ corresponds to convexity.
\end{assumption}

\subsection{Contributions}

\begin{table}
\begin{center}
\footnotesize
	\makebox[\textwidth][c]{
\begin{tabular}{|c|c|c|c|c|c|c|}
\hline
Reference 
& \begin{tabular}{c}Fast \\ scvx \\ rate \end{tabular} 
&  \begin{tabular}{c} Works for  \\ $\nicefrac{L}{\mu}\ge n$ \end{tabular} 
&  \begin{tabular}{c} Stepsize: \\ $\mu$ not \\ needed  \end{tabular} 
&  \begin{tabular}{c} minibatch \end{tabular} 
& \begin{tabular}{c} cvx  \\  rate \end{tabular} 
& \begin{tabular}{c} ncvx \\ rate \end{tabular} \\
\hline
\begin{tabular}{c}
 Mairal 2013~\cite{mairal2013optimization}\end{tabular} 
& \xmark & \cmark & $\nicefrac{1}{L}$\;  \cmark & \xmark  & \cmark & \xmark \\
\hline
\begin{tabular}{c}
 Mairal 2013~\cite{mairal2013optimization}\end{tabular} 
   & \cmark & \xmark &  $\nicefrac{1}{\mu}$ \; \xmark & \xmark & \xmark & \xmark  \\
\hline
\begin{tabular}{c} DCD 2014~\cite{defazio2014finito}\end{tabular}
& \cmark & \xmark & $\nicefrac{1}{\mu}$ \; \xmark  & \xmark  & \xmark & \xmark \\
\hline
\begin{tabular}{c} LMH 2015~\cite{lin2015universal} \end{tabular}  &\cmark & \cmark & $\nicefrac{1}{\mu}$ \;  \xmark & \xmark & \xmark & \xmark \\
\hline 
{\bf  THIS WORK} & \cmark & \cmark & $\nicefrac{n}{L}$ \; \cmark  & \cmark  & \cmark & \cmark \\
\hline
\end{tabular}
}
\caption{Summary of  contributions (svcx = strongly convex; cvx = convex; ncvx = non-convex; $\mu$ = strong convexity parameter; $L$ = smoothness parameter).
\label{tbl:contribution}
}
\end{center}
\end{table}

To break the above mentioned limits of the MM approach, we propose in this work a {\em new version of MISO}, and equip it with a  {\em more powerful convergence analysis}. Our contributions are  summarized in Table~\ref{tbl:contribution}. The comparison of the convergence rates can be found in Table \ref{tbl:rate} in Appendix \ref{sec:table3}. 

$\bullet$ {\bf First minibatch MISO.} We develop the {\em first minibatch variant of MISO}, one which interpolates between standard MISO and gradient descent.  

$\bullet$ We show that MISO can be run with a  {\em stepsize independent  of the strong convexity parameter}, which is typically unknown or hard to estimate. Unlike some versions of MISO, our variant works both in the big ($n\geq  \nicefrac{L}{\mu}$) and small ($n\leq\nicefrac{L}{\mu}$)  data regimes.   

$\bullet$ {\bf Strongly convex case.} In the strongly convex case, and for the minibatch size $\tau=n$, i.e., in the batch case, we obtain the $2 \cdot \nicefrac{ L_f}{\mu} \cdot \log \nicefrac{1}{\epsilon}$ rate of gradient descent~\cite{NesterovBook}. On the other hand, for $\tau=1$ the bound becomes $2\max\left\{n,  \nicefrac{L_f}{\mu} + \nicefrac{6L}{\mu} \right\} \log \nicefrac{1}{\epsilon}$, which is the same rate, up to small constants, as the rate of other known (non-accelerated) variance reduced methods, such as SDCA~\cite{SDCA, shalev2016sdca}, SVRG~\cite{SVRG}, S2GD~\cite{S2GD}, SAGA~\cite{SAGA} and L-SVRG~\cite{L-SVRG}. 

$\bullet$ {\bf Convex case.} Unlike all except one variant of MISO~\cite{mairal2013optimization}, we prove a sublinear $\cO(\nicefrac{L}{\epsilon})$ complexity in the convex case improving upon $\cO(\nicefrac{n L}{\epsilon})$ complexity of~\cite{mairal2013optimization}. 

$\bullet$ {\bf Nonconvex case} Finally, we give the first complexity analysis in the nonconvex case, matching the $\cO ( \nicefrac{n^{\nicefrac{2}{3}} L}{\epsilon} )$ bound of a rather complicated variant of SVRG~\cite{Reddi2016}. In contrast, MISO is much simpler and does not need to be adjusted to enjoy a good rate in the nonconvex setting.

\section{Minibatch Selection} \label{sec:mini}

As we shall see in the next section, in a key step of our method we choose a random subset $S\subseteq [n]\eqdef \{1,2,\dots,n\}$ (a ``sampling''), independently, from a user-defined distribution $\cP$ over all $2^n$ subsets of $[n]$. This distribution can be seen as a parameter of the method.  The following notions will be useful.

\begin{definition} We say that $\cP$ is {\em proper} if   $p_i \eqdef \mathbb{P}(i \in S) >0$ for all $i$. We say that $\cP$ is {\em uniform} if  $p_i=p_j$ for all $i,j$. The (expected) {\em minibatch size} of $S\sim \cP$ is the quantity $\tau  \eqdef \mathbb{E} [|S|]$.
\end{definition}

The simplest choice of a proper uniform distribution is to define $\cP$ by assigning probability $\tfrac{1}{n}$ to all single-element subsets $\{i\}$ of $[n]$. Slightly more generally, we can choose a fixed minibatch size $\tau \in [n]$, and define $\cP$ as the uniform distribution over subsets of $[n]$ of cardinality $\tau$ (this is called ``$\tau$-nice sampling''~\cite{PCDM}). However, we will perform our complexity analysis for any proper uniform distribution.

\begin{assumption} \label{as:uniform}
Distribution $\cP$ over subsets of $[n]$ used in Algorithm~\ref{sec:ALG} is proper and uniform. 
\end{assumption}

Given a proper uniform $\cP$,   it is easy to see that  $p_i = \frac{\tau}{n}$ for all $i$. Besides the expected minibatch size $\tau$, the distribution $\cP$ will enter our iteration complexity guarantees and influence the stepsize selection rule through two additional constants, $\cA \geq 0$ and $\cB\geq 0$, defined next.


\begin{assumption}\label{as:S}
Assuming $\cP$ is proper, let  $\cA\geq 0$ and $\cB \geq 0$ be such constants that  the inequality
	\begin{equation}\label{eq:sampling_asumption}
\compactify	\mathbb{E}_{S\sim \cP} \left[ \norm{ \sum \limits_{i\in S} \frac{a_i}{p_i} }^2 \right] \leq \cA\sum  \limits_{i=1}^n\|a_i\|^2 + \cB \norm{ \sum  \limits_{i=1}^na_i }^2
	\end{equation}
holds	for all vectors $a_1,\dots,a_n\in \R^d$. 
\end{assumption}

Our next result addresses the question of existence of constants $\cA$ and $\cB$. 

\begin{lemma}[Existence] \label{lm:CDexist}  For any proper  $\cP$, there are constants $\cA, \cB\geq 0$ satisfying \eqref{eq:sampling_asumption}. 
\end{lemma}

In view of the above result, Assumption~\ref{as:S} is not an assumption on the {\em availability} of constants $\cA,\cB$. Instead, the assumption just says that the constants need to be large enough. These constants will be used to set the stepsize, and will also appear in our complexity estimates. Below we compute these constants for the $\tau$-nice sampling.

\begin{lemma}\label{lm:tauniceCD}
If $\cP$ is the  $\tau$-nice sampling, then Assumption~\ref{as:S} holds  as long  as $\cA \geq \frac{n(n-\tau)}{\tau(n-1)}$ and $\cB \geq \frac{n(\tau-1)}{\tau(n-1)}$. 
\end{lemma}

\section{The Algorithm} \label{sec:ALG}

Our proposed method---Minibatch MISO (Algorithm~\ref{alg:miso})---is a generalized variant of the incremental MISO algorithm~\cite{mairal2013optimization}. 

The algorithm is initiated with auxiliary vectors $\phi_1^0, \phi_2^0, \dots, \phi_n^0\in \R^d$ which can take any values,  after which we compute the gradients $f^{\prime}(\phi_i^0)$ for all $i$ and set 
\[ \compactify x^{k} = \bar \phi^{k} - \frac{\gamma}{n}\sum \limits_{i=1}^n f_i'(\phi_i^{k})\] 
for $k=0$, where $\bar\phi^k$ is the average of all $\phi_i^k$. In fact, this relation will be used throughout to define the main optimization step of the method. In Step 4  we sample a set $S^k\sim \cP$. In Step 5, only auxiliary vectors $\phi_i^k$ for $i\in S^k$ are updated (to $x^k$); the rest are kept unchanged.  In Step 6 we maintain the average of the auxiliary vectors. Note that this can be done at cost $\cO(d |S^k| )$ arithmetic operations.  The key optimization step is Step 7, where we take a gradient-type step from the average vector $\bar \phi^{k+1} $, with stepsize $\gamma$. 

\begin{algorithm}[t]
   \caption{Minibatch MISO}
   \label{alg:miso}
\begin{algorithmic}[1]
\State {\bf Parameters:} Stepsize $\gamma>0$; initial auxiliary vectors $\phi_1^0, \dotsc, \phi_n^0 \in \R^d$; distribution $\cP$ over subsets of $\{1,2,\dots,n\}$
\State {\bf Initialize:} initialize average $\overline \phi^0 = \frac{1}{n}\sum_{i=1}^n \phi_i^0$ and starting point $x^0 = \overline \phi^0 - \frac{\gamma}{n}\sum_{i=1}^n f_i'(\phi_i^0)$  
   \For{$k=0,1,\dotsc$}
       \State Choose a random subset $S^k \subseteq \{1, \dotsc, n\}$ according to distribution $\cP$
       \State $\phi_i^{k+1} = \begin{cases} x^k & \quad i\in S^k \\
       \phi_i^k & \quad i\notin S^k \end{cases}$ \hfill $\triangleright$ update a random subset of auxiliary vectors
       \State $\bar \phi^{k+1} = \frac{1}{n}\sum_{i=1}^n \phi_i^{k+1}$ \hfill $\triangleright$ maintain average 
       \State $x^{k+1} = \bar \phi^{k+1} - \frac{\gamma}{n}\sum \limits_{i=1}^n f_i'(\phi_i^{k+1})$ \hfill $\triangleright$ take step
   \EndFor
\end{algorithmic}
\end{algorithm}

In the  lemma below, we show that the difference between two consecutive iterates $x^k$ and $x^{k+1}$ points in the gradient direction  $f^{\prime}(x^k)$, on average.	Throughout the paper, we use $\mathbb{E}_k[\cdot]$ to denote the conditional expectation on $x^k$ and $\phi^k_i$.

\begin{lemma} \label{lm:???} The iterates of Algorithm~\ref{alg:miso} for all $k\geq 0$ satisfy
	\begin{equation}\label{eq:xiugd97g9f}
\compactify	\mathbb{E}_k[x^{k+1} - x^k] = -\frac{\gamma\tau}{n}f^{\prime}(x^k)\;.
	\end{equation}

\end{lemma}

Hence, when viewed through the lens of the $x$ iterates only, Algorithm~\ref{alg:miso} can be seen as an instance of  stochastic gradient descent. Besides offering this insight, this lemma will be used to prove Theorem~\ref{Lem:key}, which plays a key role in the convergence analysis in the convex and strongly convex case.

 \section{Convergence Theory for Convex and Strongly Convex $f$} \label{sec:convex_and_str_convex}

The ``error'' quantity 
$
\compactify {\cal W}^k \eqdef \sum \limits_{i=1}^n\|\phi_i^{k} - x^* - \gamma f^{\prime}_i(\phi_i^{k}) + \gamma f^{\prime}_i(x^*)\|^2,
$
where $x^*$ is any fixed element of $\cX^*$,  plays a key role in our complexity results. First, we show that $\cW^k$ satisfies a recursion, relating its evolution to the distance between $x^k$ and $x^*$, and suboptimality gap $f(x^k)-f(x^*)$.

\begin{lemma}\label{lm:Dk+1} Assume that $f$ is convex and that functions $f_i$ are $L$-smooth. Let $\cP$ satisfy Assumption~\ref{as:uniform}. Then the iterates  of Algorithm~\ref{alg:miso}  satisfy the relation
	\[
	\compactify 
	\mathbb{E}_k[{\cal W}^{k+1}] \leq \left(1-\frac{\tau}{n}\right){\cal W}^k + \tau\|x^k - x^*\|^2 + 2L\tau\gamma^2 (f(x^k) - f(x^*)). 
	\]
\end{lemma}

The above lemma can be used to establish the following key technical result.

\begin{lemma}\label{Lem:key} Assume that $f$ is convex and $L_f$-smooth,  and that functions $f_i$ are convex and $L$-smooth. Let $\cP$ satisfy Assumption~\ref{as:uniform}, and let $\cA,\cB$ be chosen as in Assumption~\ref{as:S}. For $p\geq 0$, define the Lyapunov function  
\[
\compactify
\Psi^k_p \eqdef  \|x^{k} - x^*\|^2 + \frac{(2+p)\cA\tau}{n^3}{\cal W}^{k}. 
\] 
Then  the iterates  of Algorithm~\ref{alg:miso}  satisfy the relation
	\begin{eqnarray*}
	\mathbb{E}_k \left[ \Psi^{k+1}_p  \right]	 
		&\leq& \compactify\left(  1 - \frac{\tau}{n} \left(\gamma \mu - \frac{\cA\tau p}{n^2} \right) \right)\|x^k -x^*\|^2 + \tfrac{\cA\tau (2+p)}{n^3} \left(  1 - \frac{\tau}{n}\cdot \frac{p}{2+p}  \right) {\cal W}^k  \\ 
		&&\compactify \quad - \frac{2\gamma\tau}{n}\left( 1 - \frac{\tau\gamma}{n}\left(\cB L_f + \frac{(4+p)\cA L}{n} \right)   \right) \left(f(x^k) - f(x^*)\right) \;. 
	\end{eqnarray*}	
\end{lemma}

Lemma~\ref{Lem:key} provides a key step in the proofs of our main convergence theorems, Theorem~\ref{Th:stronglyconvex}~and~\ref{Th:gconvex}, described next.

\subsection{Strongly Convex Case}

Our main result in the strongly convex case ($\mu>0$),  stated next,  posits a linear convergence rate.

\begin{theorem}\label{Th:stronglyconvex}   Assume that $f$ is $\mu$-strongly convex with $\mu>0$ and $L_f$-smooth,  and that functions $f_i$ are convex and $L$-smooth. Let Assumptions~\ref{as:uniform}~and~\ref{as:S} hold. 	Let ${\cal L} \eqdef  \cB L_f+ \frac{6\cA L}{n} $ and $\gamma = \frac{n}{\tau {\cal L}}$. Then  for all $k\geq 0$ we have
	\begin{equation}\label{eq:b9f7g9f}
	\compactify \mathbb{E}[\|x^k - x^*\|^2] \leq \left( 1 - \min\left\{ \frac{\tau}{2n}, \frac{\mu}{2{\cal L}}, \frac{\mu n^2}{8\cA \tau {\cal L}}  \right\}  \right)^k \left( \|x^0-x^*\|^2 + \frac{4\cA\tau}{n^3}{\cal W}^0 \right). 
	\end{equation}
\end{theorem}

The above result is perhaps best understood by fixing a target error tolerance $\epsilon$, and using \eqref{eq:b9f7g9f} to find a bound on $k$ for which this error tolerance is guaranteed. Standard computations lead to the following corollary.

\begin{corollary}\label{co:stronglyconvexns}
Let $\cP$ be the $\tau$-nice sampling, and let $\gamma = \frac{n}{\tau L_f + 6L(n-\tau)/(n-1)}$. Assume $n\geq 4$. Then 
$$
\compactify k \geq 2 \max \left\{  \frac{n}{\tau}, \ \frac{L_f}{\mu} + \frac{6L}{\tau \mu}\cdot \frac{n-\tau}{n-1}  \right\} \log\left(  \frac{\|x^0-x^*\|^2 + \frac{4\cA\tau}{n^3}{\cal W}^0}{\epsilon}  \right) \quad \Rightarrow \quad \mathbb{E}[\|x^k - x^*\|^2] \leq \epsilon. 
$$
\end{corollary}

Note that the number of iterations decreases as $\tau$ increases.  For $\tau=n$, i.e., in the batch case, the bound becomes $2 \cdot \nicefrac{ L_f}{\mu} \cdot \log \nicefrac{1}{\epsilon}$, which is the rate of gradient descent~\cite{NesterovBook}. On the other hand, for $\tau=1$ the bound becomes $2\max\left\{n,  \nicefrac{L_f}{\mu} + \nicefrac{6L}{\mu} \right\} \log \nicefrac{1}{\epsilon}$, which is the same rate, up to small constants, as the rate of other known (non-accelerated) variance reduced methods, such as SDCA~\cite{SDCA, shalev2016sdca}, SVRG~\cite{SVRG}, S2GD~\cite{S2GD}, SAGA~\cite{SAGA} and LSVRG~\cite{L-SVRG}. This second bound is always worse than the first, especially in the big data regime (i.e., when $n\gg \nicefrac{L_f}{\mu}$), or when $L\gg L_f$. However, if $L=\cO(L_f)$ and if the condition number satisfies $\nicefrac{L_f}{\mu} = \cO(n)$, the two bounds are identical, up to a constant factor.
The general minibatch case interpolates between these extremes. 

\subsection{Convex Case}

Our main result in the convex case ($\mu=0$)  offers a $\cO(\nicefrac{1}{k})$ convergence rate.

\begin{theorem}\label{Th:gconvex} Assume that $f$ is convex and $L_f$-smooth,  and that functions $f_i$ are convex and $L$-smooth. Let Assumptions~\ref{as:uniform}~and~\ref{as:S} hold.  Choose stepsize $\gamma = \frac{n}{2\tau(\cB L_f + 4\cA L/n)}$. Then for $x^a$ chosen uniformly at random from $\{  x^i  \}_{i=0}^k$ we have 
$$
\compactify
\mathbb{E}[f(x^a)] - f(x^*) \leq 2\left(\cB L_f + \frac{4\cA L}{n}\right) \frac{ \left(\|x^0 - x^*\|^2 + \frac{2\cA \tau}{n^3}{\cal W}^0 \right)}{k+1}. 
$$

\end{theorem}

To shed more light on this rate, in the next corollary we specialize this result to the case of the $\tau$-nice sampling, for which formulas for $\cA$ and $\cB$ are readily available from Lemma~\ref{lm:tauniceCD} (we choose $\cA = \frac{n(n-\tau)}{\tau(n-1)}$ and $\cB = 1$).

\begin{corollary}
 Let $\cP$ be the $\tau$-nice sampling, and let $\gamma = \frac{n}{2\tau(L_f + \frac{4L}{\tau}\cdot \frac{n-\tau}{n-1})}$. If $x^a$ is chosen uniformly at random from $\{  x^i  \}_{i=0}^k$,  then 
$$
\compactify
k \geq 2\left(  L_f + \frac{4L}{\tau} \cdot \frac{n-\tau}{n-1}  \right) \frac{ \left(\|x^0 - x^*\|^2 + \frac{2\cA \tau}{n^3}{\cal W}^0 \right)}{\epsilon} \quad \Rightarrow \quad \mathbb{E}[f(x^a)] - f(x^*) \leq \epsilon \;. 
$$
\end{corollary}

Yet again, the number of iterations decreases as $\tau$ increases. For $\tau=n$, i.e., in the batch case, we have $\cA =0$, the bound becomes $2L_f \|x^0 - x^*\|^2\frac{1}{\epsilon} $, which is the standard rate of gradient descent~\cite{NesterovBook}. Fro $\tau=1$, on the other hand, we get $\cA =n$, and the bounds simplifies to $2(L_f+4L)(\|x^0 - x^*\|^2 + \frac{2}{n^2}\cW^0)  \frac{1}{\epsilon}$, which is  the same $\cO(\nicefrac{L}{\epsilon})$ rate of other variance reduced methods in this regime.

\subsection{More commentary}

In summary, both in the convex and strongly convex cases, our rate for minibatch MISO interpolates between the rate of gradient descent and state-of-the-art rates of other more popular (non-accelerated)  variance reduced methods. This closes a gap in the literature. Moreover, and perhaps more importantly, unlike \cite{defazio2014finito, lin2015universal}, our step size selection rules in the strongly convex case do not depend on the knowledge of the strong convexity parameter $\mu$, which is often hard to estimate.

\section{Convergence Theory for Nonconvex $f$} \label{sec:non-convex}

In this section we establish iteration complexity bounds for Minibatch MISO without any convexity assumptions. Our goal will be to find a point with a small gradient. We establish  the first rates for a MISO-type method in the nonconvex setting. 

\subsection{Technical lemmas}

The following two technical lemmas play a key role in our analysis. The first result provides a bound on the distance of two consecutive iterates.

\begin{lemma}\label{lm:dxnonconvex} Assume that functions $f_i$ are  $L$-smooth. Let Assumptions~\ref{as:uniform}~and~\ref{as:S} hold. Then
$$
\compactify 
\mathbb{E}_k[\|x^{k+1} - x^k\|^2] \leq \frac{2\tau^2\cA (1+\gamma^2L^2)}{n^3} \cdot \frac{1}{n}\sum \limits_{i=1}^n \|x^k - \phi^k_i \|^2 + \frac{\tau^2\gamma^2\cB}{n^2}\|f^{\prime}(x^k)\|^2. 
$$
\end{lemma}

The quantity $\mathbb{E}^i[|S|] \eqdef \mathbb{E}[|S| \;|\; i\in S]$ is the expected minibatch size of minibatches which contain $i$. In the rest of this section, let $M \eqdef \max \limits_{1\leq i\leq n}\mathbb{E}^i[|S|]$.  Our second technical lemma gives a bound on the average distance between the iterates and the auxiliary variables.

\begin{lemma}\label{lm:phik+1nonconvex} Let Assumption~\ref{as:uniform} hold. Then for any $\beta>0$, we have 
\begin{eqnarray*}
\compactify
\mathbb{E}_k\left[ \frac{1}{n} \sum \limits_{i=1}^n\|x^{k+1} - \phi^{k+1}_i \|^2 \right] &\leq& 
\compactify \left(  \frac{6M}{\tau} +1  \right) \mathbb{E}_k[\|x^{k+1} - x^k\|^2] + \frac{\gamma \tau}{n\beta} \|f^{\prime}(x^k)\|^2  \\ 
&& \compactify \quad + \left( 1 - \frac{\tau}{n} + \frac{\tau^2}{6n^2} + \frac{\gamma \tau \beta}{n}  \right) \frac{1}{n} \sum \limits_{i=1}^n \|x^k - \phi^k_i\|^2.  
\end{eqnarray*}

\end{lemma}

\subsection{Main result}

We are now ready to state our main convergence result in the nonconvex case.

\begin{theorem}\label{Th:nonconvex1}
	Let $f$ be $L_f$-smooth and $f_i$ be $L$-smooth. Let Assumptions~\ref{as:uniform}~and~\ref{as:S} hold.  Assume $\frac{n^2}{\tau \cA} \geq 24\left(\frac{6M}{\tau} + 1\right)$. Consider the Lyapunov function
	$$
	\compactify
	\Psi^k \eqdef f(x^k) + \alpha \cdot \frac{1}{n} \sum \limits_{i=1}^n \|x^k - \phi^k_i\|^2, 
	$$
	where $\alpha = \frac{\beta}{q}$, $q = \max\{  4, 2\cB (\nicefrac{6M}{\tau} +1)  \}$, and $\beta = \frac{1}{2\gamma}$. If the stepsize satisfies 
	\begin{equation}\label{eq:gammanonconvex}
\compactify	
	\gamma \leq \min \left\{ \frac{n}{2\cB L_f\tau}, \frac{n^2/\tau \cA}{24qL_f},  \frac{(n^2/\tau \cA)^{\frac{1}{3}}}{(24qL_fL^2)^{\frac{1}{3}}} , \frac{(n^2/\tau \cA)^{\frac{1}{2}}}{(24(6M/\tau +1)L^2)^{\frac{1}{2}}}  \right\}, 
	\end{equation}
	then 
$	
	\mathbb{E}_k [\Psi^{k+1}] \leq \Psi^k - \frac{\gamma \tau}{4n}\|\nabla f(x^k) \|^2. 
$
\end{theorem}

While the above result does not spell out the rate explicitly, it provided an easy to analyze recursion, which leads to more interpretable corollaries.  The first one  gives an $\cO(\nicefrac{1}{\epsilon})$ rate for any minibatch strategy $\cP$ satisfying Assumption~\ref{as:S}.

\begin{corollary}\label{co:nonconvex1}
	Let $x^a$ be chosen uniformly at random from $\{  x^i  \}_{i=0}^k$ and  $\gamma$ satisfy (\ref{eq:gammanonconvex}). Assume $\frac{n^2}{\tau \cA} \geq 24(\frac{6M}{\tau} + 1)$. 
	Then 
	$$
	\compactify	
	\mathbb{E}[\|\nabla f(x^a)\|^2 ] \leq \frac{4n}{\gamma \tau}\cdot \frac{f(x^0) - f(x^*)}{k+1}. 
	$$
	If $\gamma$ is equal to the upper bound in (\ref{eq:gammanonconvex}), then $\mathbb{E}[\|\nabla f(x^a)\|^2 ]  \leq \epsilon$ as long as 
	$$
	\compactify	
	k \geq \left(  \frac{4n}{\tau} \max \left\{\frac{2\cB L_f\tau}{n}, \frac{24qL_f}{n^2/\tau \cA},  \frac{(24qL_fL^2)^{\frac{1}{3}}}{(n^2/\tau \cA)^{\frac{1}{3}}} , \frac{(24(6M/\tau +1)L^2)^{\frac{1}{2}}}{(n^2/\tau \cA)^{\frac{1}{2}}}    \right\}  \frac{f(x^0) - f(x^*)}{\epsilon}  \right). 
	$$
\end{corollary}

We now specialize the above result to $\tau$-nice sampling. In view of Lemma~\ref{lm:tauniceCD}, we can choose $\cA = \frac{n(n-\tau)}{\tau(n-1)}$ and $\cB =1$. Also, we have $M = \tau$ for $\tau$-nice sampling. Hence, from Corollary \ref{co:nonconvex1}, we can obtain the following result. 

\begin{corollary}\label{co:nonconvex2}
For $\tau$-nice sampling, let $x^a$ be chosen uniformly at random from $\{  x^i  \}_{i=0}^k$ and  $\gamma$ be equal to the upper bound in (\ref{eq:gammanonconvex}) with $q = 14$, $\cA = \frac{n(n-\tau)}{\tau(n-1)}$, and $\cB =1$. Assume $n \geq 168$. Then $\mathbb{E}[\|\nabla f(x^a)\|^2 ]  \leq \epsilon$ as long as 
$$
\compactify
k \geq {\cal O} \left( \left(  L_f + \frac{n^{\frac{2}{3}} (L_fL^2)^{\frac{1}{3}}}{\tau} \cdot \frac{(n-\tau)^{\frac{1}{3}}}{(n-1)^{\frac{1}{3}}}  + \frac{\sqrt{n}L}{\tau} \cdot \frac{(n-\tau)^{\frac{1}{2}}}{(n-1)^{\frac{1}{2}}}\right) \frac{f(x^0) - f(x^*)}{\epsilon} \right). 
$$
\end{corollary}

For $\tau=n$, the above rate simplifies to $L_f \frac{f(x^0) - f(x^*)}{\epsilon}$, which is the rate of gradient descent.  For $\tau=1$, the rate simplifies to
$\left(L_f+ n^{\nicefrac{2}{3}}\left(L_f L^2\right)^{\nicefrac{1}{3}} + n^{\nicefrac{1}{2}}L \right)  \frac{f(x^0) - f(x^*)}{\epsilon} = \cO\left( \frac{n^{\nicefrac{2}{3}} L}{\epsilon} \right)$. This is the same rate as the rate of a (complicated) variant of SVRG~\cite{Reddi2016}. In contrast, MISO is much simpler and does not need to be adjusted to enjoy a good rate in the nonconvex setting.

\section{Experiments}\label{sec:exp}
In this section we run experiments and show performance of the minibatch MISO on real datasets. Firstly, we show how minibatch size $\tau$ affects the convergence of the minibatch MISO, and in the second part we compare the minibatch MISO with well-known minibatch variance reduced algorithms - minibatch SAGA and minibatch SVRG. Our experiments is performed on the regularized logistic regression problem:
\begin{align} \label{eq:logistic_regression}
\compactify \min_{x\in \R^d} f(x) = \frac{1}{n} \sum \limits_{i=1}^n \log (1 + \exp(-y_i \mA_{i:} x)) + \frac{\lambda}{2} \|x\|^2,
\end{align}
where $\mathbf{A} \in \R^{n \times d}$, $y \in \R^n$ and the regularization parameter $\lambda > 0$. Note that for this problem each $f_i$ is $L_i$ smooth where $L_i =  \frac{1}{4} \| \mA_{i:} \|^2 + \lambda$ and $L = \max_i L_i$, and $L_f$ = the largest eigenvalue of $\frac{1}{4n} \mathbf{A}^\top \mathbf{A} + \lambda \mI$. The problem is strongly convex with the strongly convexity constant $\mu =  \lambda$. \\

\subsection{Varying minibatch sizes} 

\begin{figure}[ht]
	\centering
	\begin{minipage}{.4\textwidth}
		\centering
		\includegraphics[width=\linewidth]{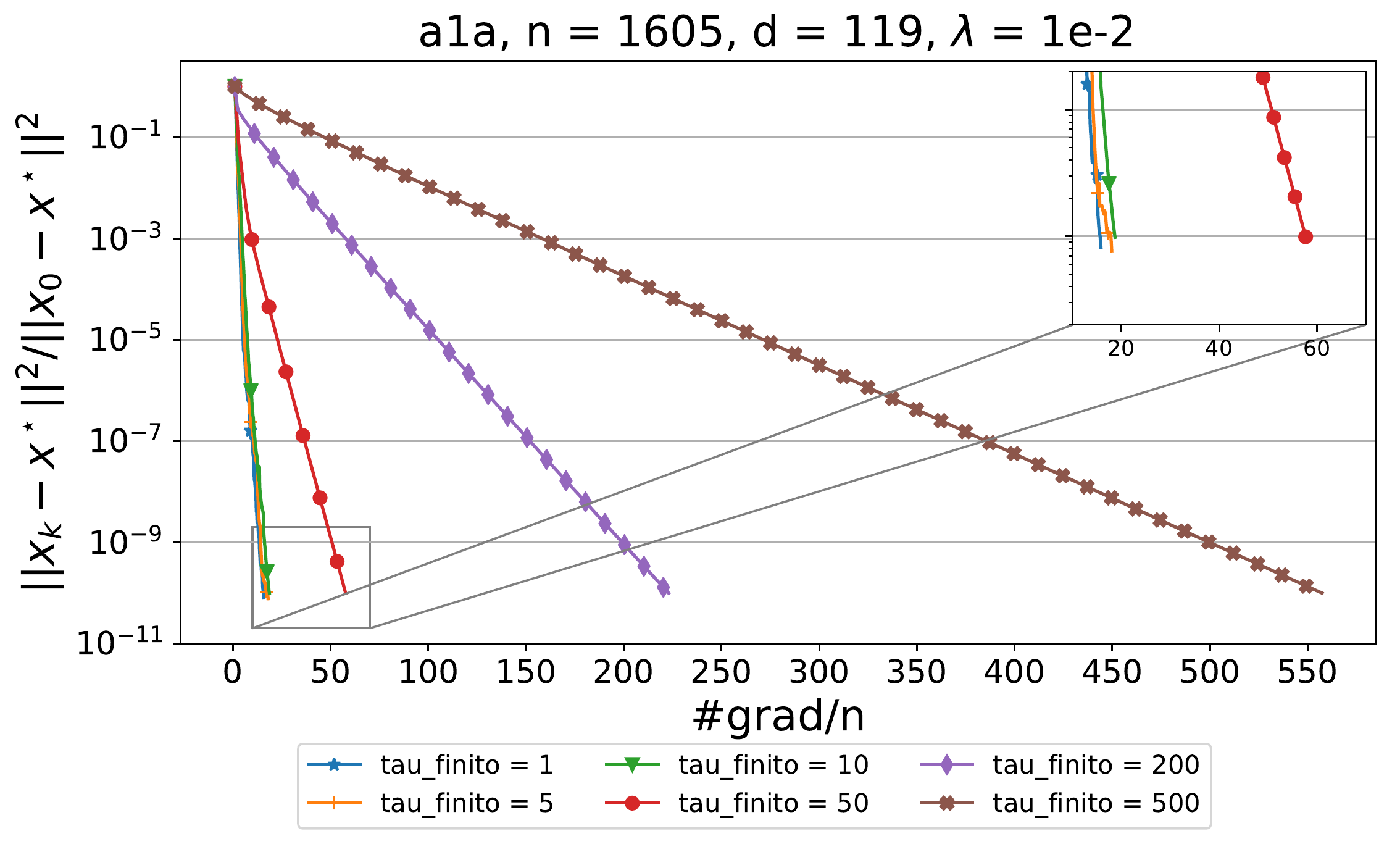}
	\end{minipage}%
	\begin{minipage}{.4\textwidth}
		\centering
		\includegraphics[width=\linewidth]{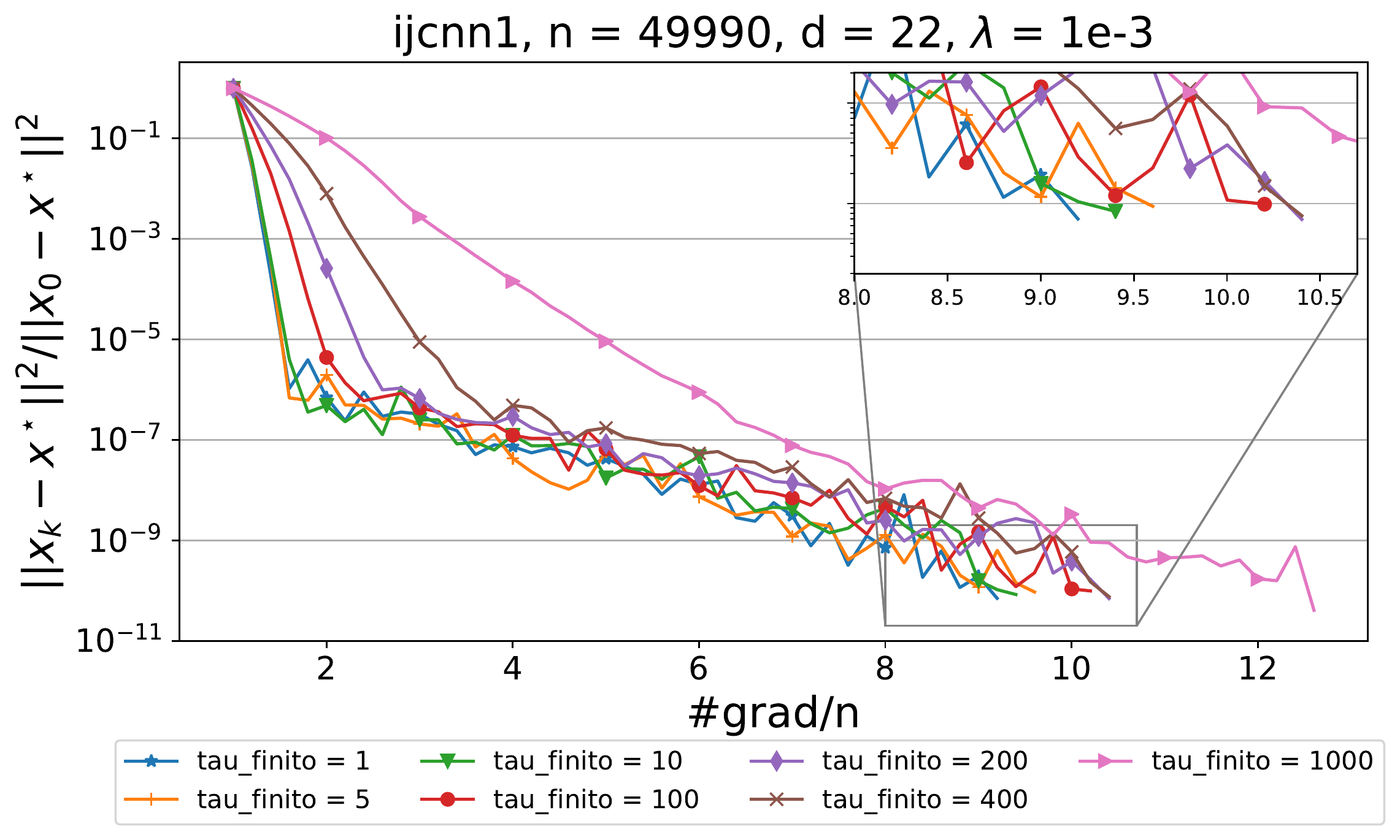}
	\end{minipage}
	\centering
	\begin{minipage}{.4\textwidth}
		\centering
		\includegraphics[width=\linewidth]{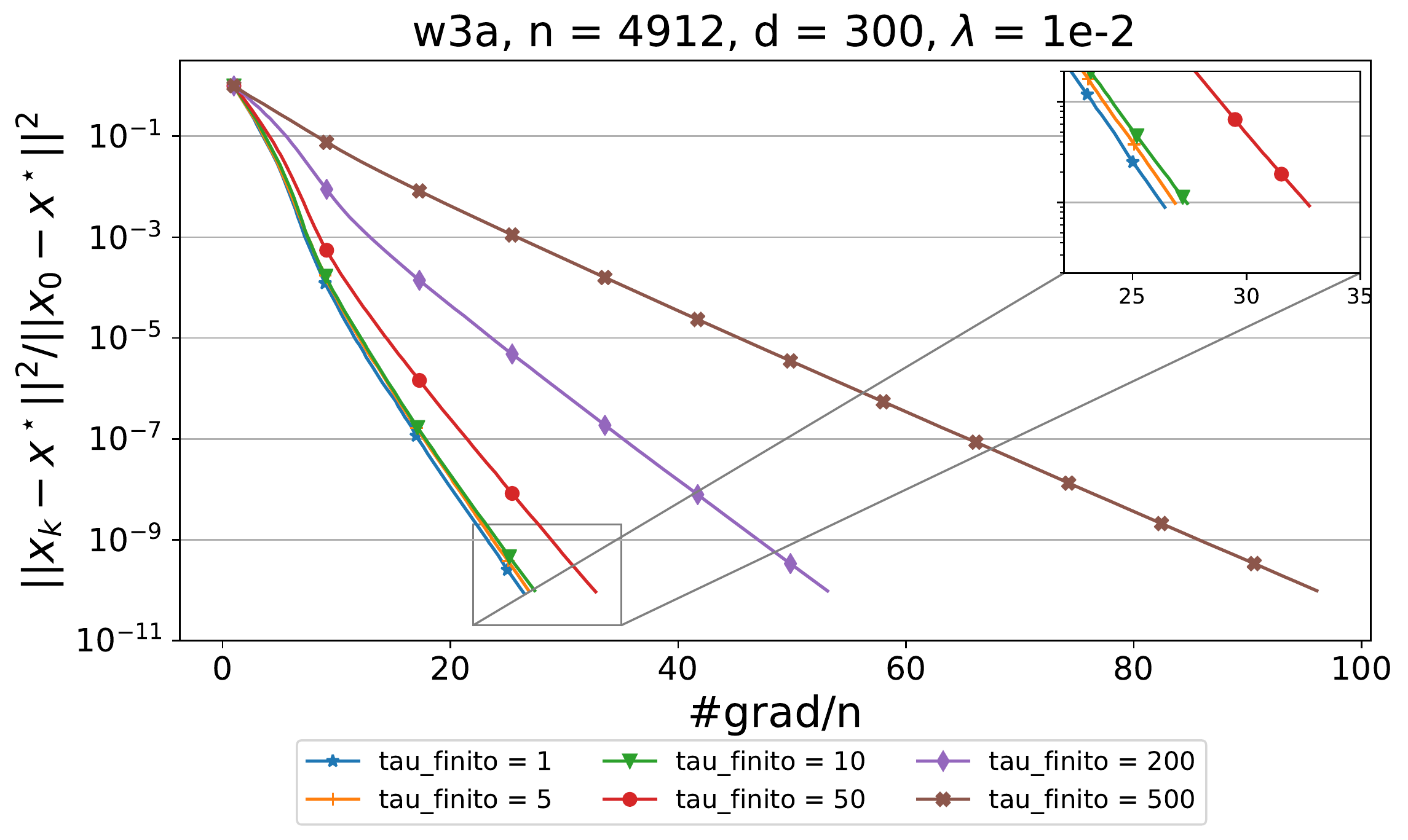}
	\end{minipage}%
	\begin{minipage}{.4\textwidth}
		\centering
		\includegraphics[width=\linewidth]{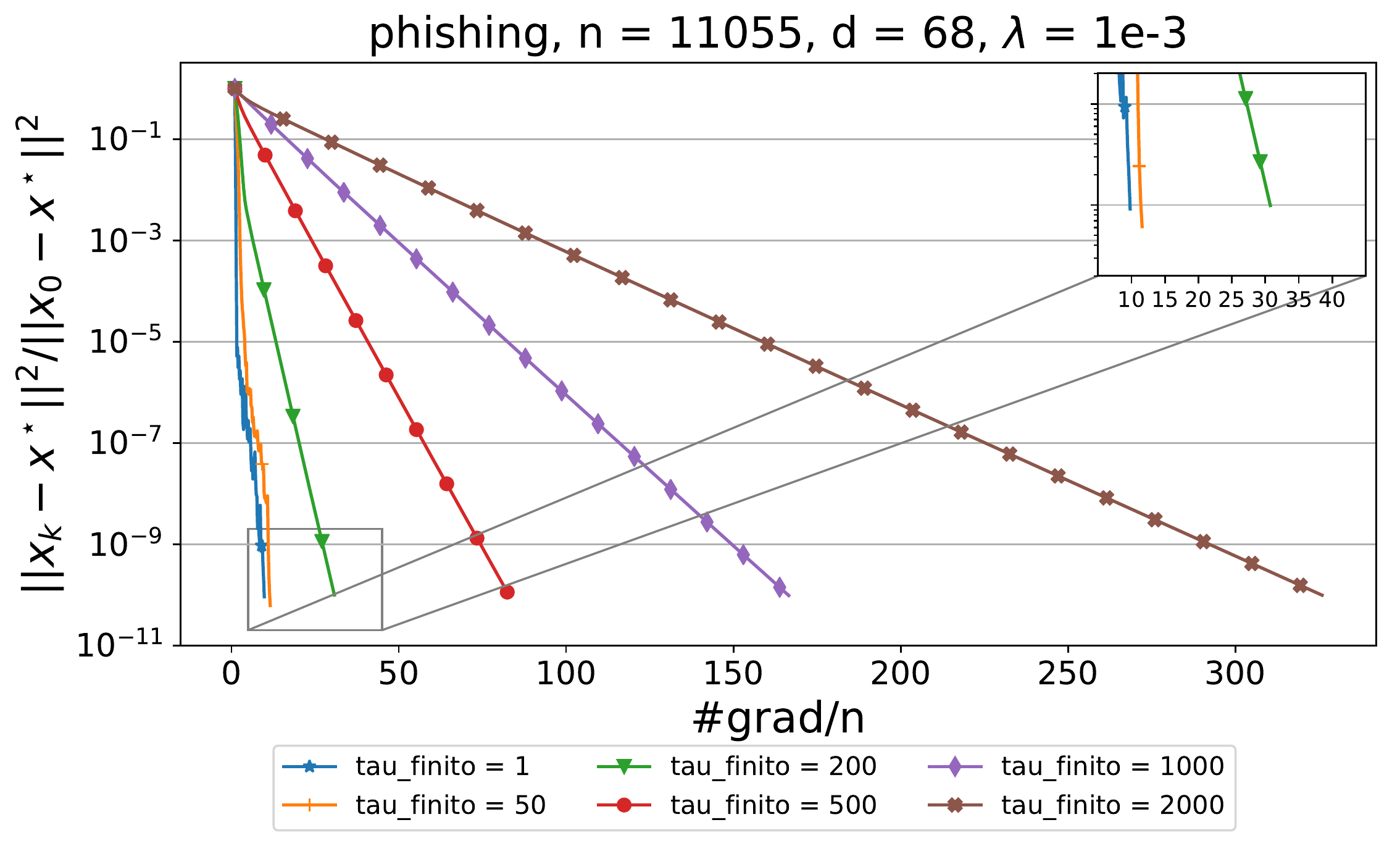}
	\end{minipage}
	\caption{The graphs show performance of the minibatch MISO on the real datasets: \texttt{ijcnn1}, \texttt{a1a}, \texttt{w3a}, \texttt{phishing}}
	\label{fig:real_minibatch}
\end{figure}

We run the minibatch MISO by choosing different batch size on several real datasets from the LIBSVM dataset \cite{LIBSVM}: \texttt{a1a}, \texttt{ijcnn1}, \texttt{w3a}, \texttt{phishing}. For every minibatch size, the stepsize $\gamma$ is chosen as given in Theorem \ref{Th:stronglyconvex}, i.e. $\gamma = \frac{n}{\tau {\cal L}}$. We run the algorithm until we get an accuracy $\frac{\| x^k - x^*\|^2}{\|x^0 - x^*\|^2} \leq 10^{-10}$, where $x^*$ is the optimum that we find by running the gradient descent algorithm on (\ref{eq:logistic_regression}), and $x^0$ is random initial point. As we have different minibatch sizes, the X-axis represents the total number of the single gradient computations divided by $n$ - for a minibatch size $\tau$ on each step the algorithm computes the single gradient $\tau$ times, given that we store the table of the gradients (Step 5-7 of the algorithm).
Figure \ref{fig:real_minibatch} shows results of the experiment. We can see that for not relatively big values of $\tau$ the number of gradients is almost the same, that means that if we run the table update in parallel, we will achieve linear speed-up in $\tau$. On the other hand, for relatively large values of $\tau$ (larger than ${\cal O}(\tfrac{L}{L_f})$), the linear speed-up can not be seen, which was expected by the theory.

\subsection{Minibatch MISO vs SAGA vs SVRG}  
\begin{figure}[!t]
	\centering
	\includegraphics[width = 0.8\linewidth]{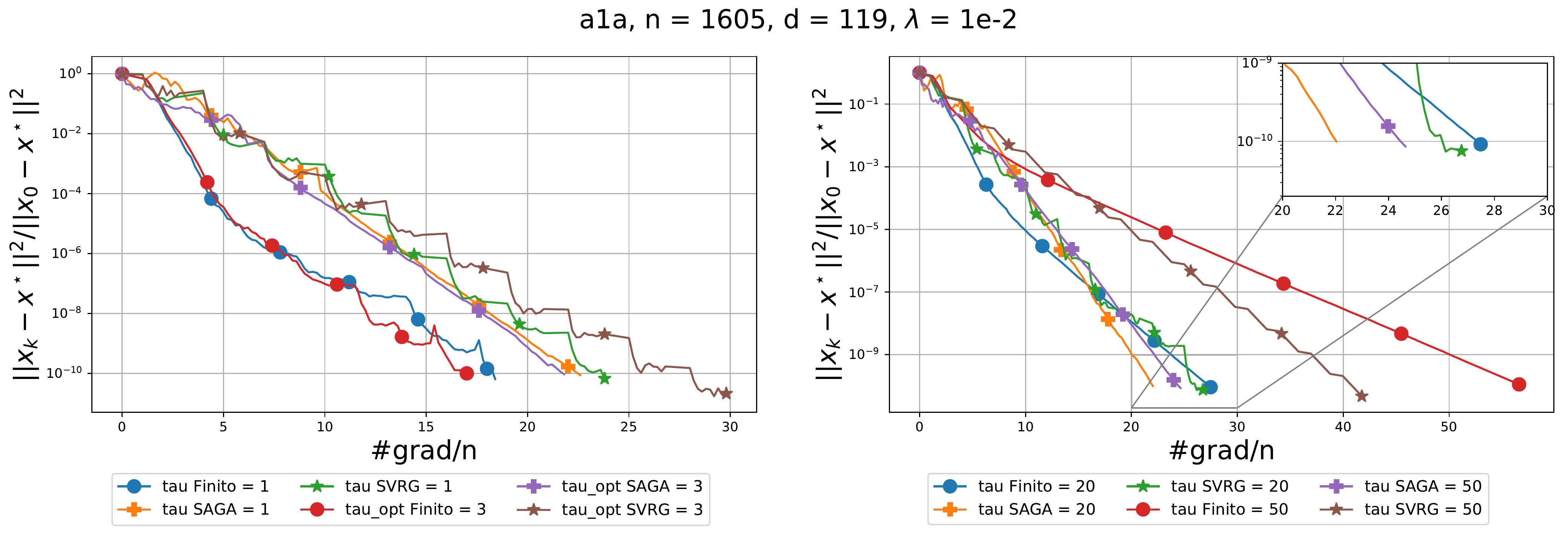}
	\includegraphics[width = 0.8\linewidth]{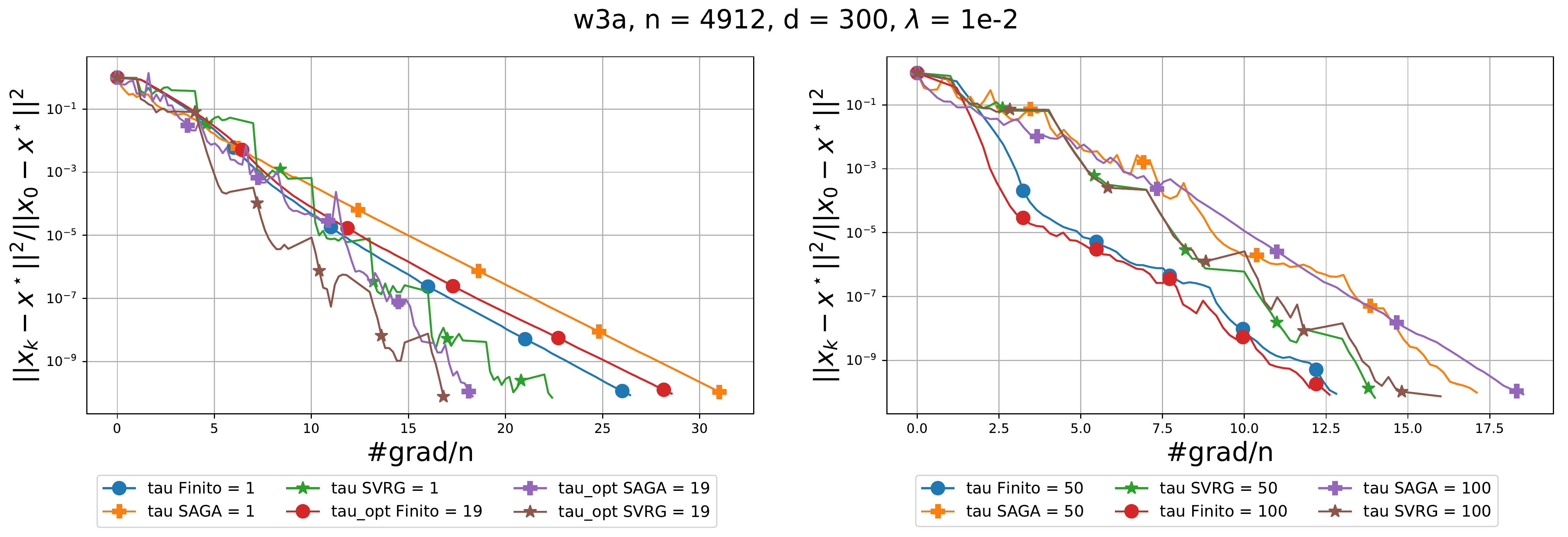}
	\caption{The convergence of the minibatch versions of MISO, SAGA and SVRG on the logistic regression problem for real datasets.}
	
		\label{fig:compare_algos1}
\end{figure}
  
In this part we compare the performances of the minibatch MISO, minibatch SAGA and minibatch SVRG. We run all algorithms on the same datasets from the LIBSVM dataset for some particular choices of $\tau$. The setting - the stopping criteria and the axis labels are the same as in the previous subsection. The important thing to take into account is the choice of the stepsize. Usually for experiments researchers run an algorithm a lot of times to find the optimal stepsize, however, in practice you are often able to run an algorithm only few times. So in our experiments, we choose the theoretical estimates for the stepsizes. To have some variation, we run the algorithms with its theoretical estimates multiplied by some factor, and then we choose the best ones. The factors we choose are $\{1, 5, 10, 20\} $. For the minibatch MISO algorithm the theoretical estimate of the stepsize is $\gamma_{\rm MISO} = \frac{n}{\tau {\cal L}}$. For the minibatch SAGA, the stepsize in theory is given by $\gamma_{\rm SAGA} = \frac{1}{4} \frac{1}{\max\{\cL + \lambda, \frac{1}{\tau} \frac{n-\tau}{n-1} L + \frac{\mu}{4} \frac{n}{\tau}\}}$ \cite{SAGA_stepsize}. For the minibatch SVRG, there is no existing theoretical estimates for parameters: the length of inner loop $m$ and the stepsize $\gamma_{\rm SVRG}$. Usually, in practice for SVRG people set $m=n$ or $m=2n$, and set the stepsize $\gamma_{\rm SVRG} = 0.1/L$ \cite{SVRG}. We tried to use the same setting for the minibatch SVRG, but the convergence was very slow for the minibatch size greater than 1. Then we tried other options, and figured out that the $\tau$ minibatch SVRG works much better for $m = [\frac{2n}{\tau}]$, similarly to increased probability of L-SVRG~\cite{horvath2019stochastic}. So in our experiments we use $m = [\frac{2n}{\tau}]$ and $\gamma_{\rm SVRG} = 0.1/L$. $tau\_opt$ in the figures are the optimal minibatch size for the minibatch SAGA given in \cite{SAGA_stepsize}. The results of our experiments are shown in Figure~\ref{fig:compare_algos1}  (also see Figure~\ref{fig:compare_algos2} in Appendix~\ref{sec:missing_fig}). Notice that the minibatch versions of the three algorithms behave in similar fashion. 
However, we can see that in the most of the experiments the minibatch MISO works better than both the minibatch SAGA and the minibatch SVRG.

\bibliographystyle{plain}
\bibliography{finito_ref.bib}

\clearpage
\appendix
\part*{Appendix}

\tableofcontents

\clearpage

\section{Figure~\ref{fig:compare_algos2}} \label{sec:missing_fig}

Here we provide a figure which we referred to in Section~\ref{sec:exp}, but which we did not include in the main paper due to space restrictions.

\begin{figure}[ht]
	\centering
	\includegraphics[width = \linewidth]{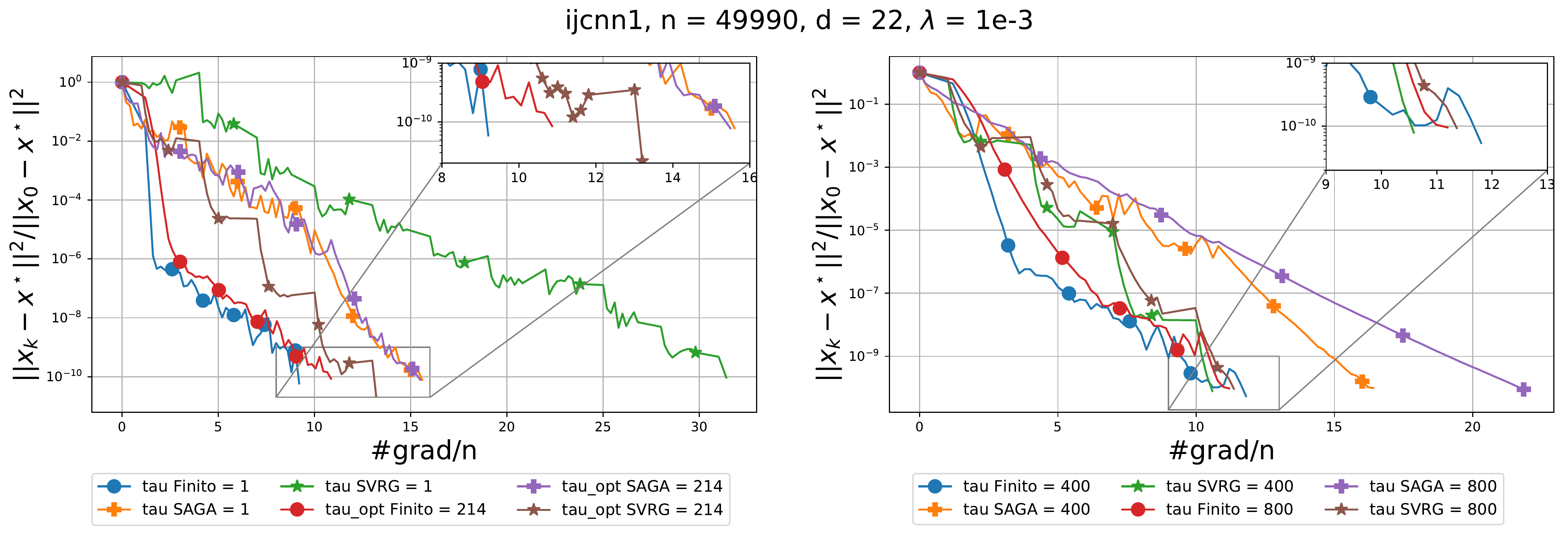}
	\includegraphics[width = \linewidth]{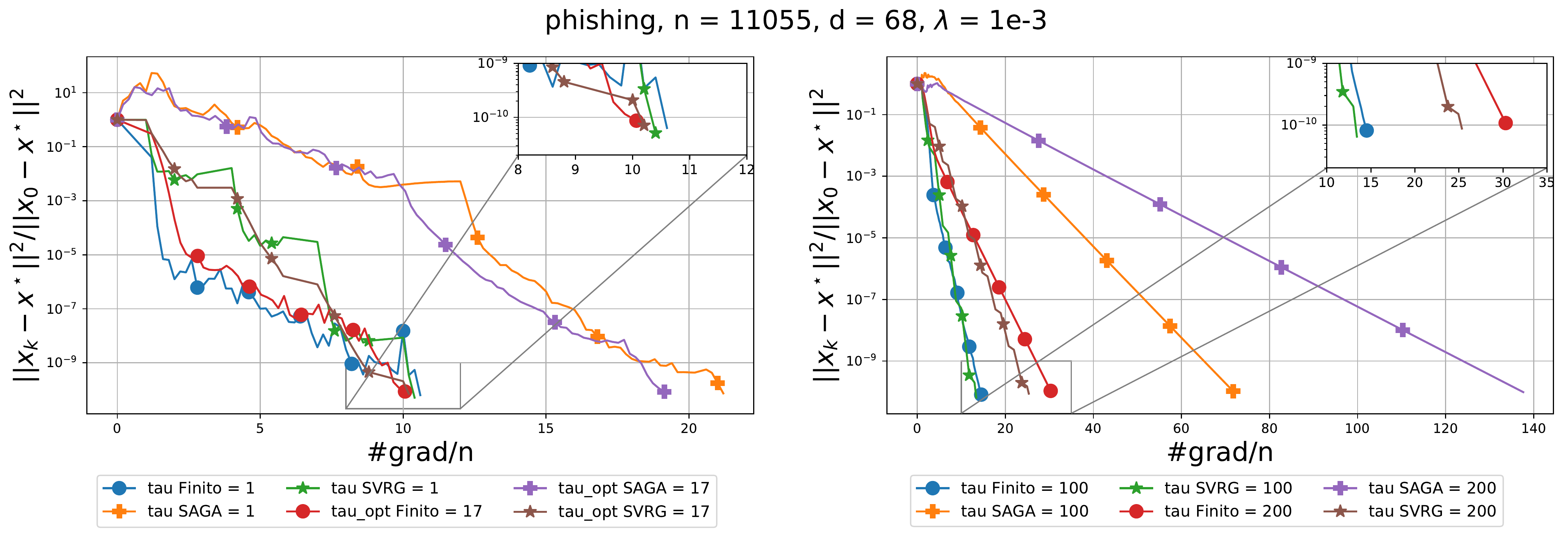}
	\caption{The convergence of the minibatch versions of MISO, SAGA and SVRG on the logistic regression problem for real datasets.}
	\label{fig:compare_algos2}
\end{figure}

\clearpage

\section{Notation Table} \label{sec:notation_appendix}

 \begin{longtable}{| p{.20\textwidth} | p{.80\textwidth}| } 
  \caption{Summary of notation used in this paper.}  \label{tbl:notation}\\
  \hline
  \multicolumn{2}{|l|}{\bf Optimization} \\
   \hline
   $f$ & objective function $f(x)=\frac{1}{n}\sum \limits_{i=1}^n f_i(x)$ \\
   $x^*$ & optimal point \\
   $L$ & Lipschitz constant of $f^{\prime}_i$ \\
   $L_f$ & Lipschitz constant of $f^{\prime}$ \\
   $\mu$ & strong convexity constant of $f$\\
     \hline
   \multicolumn{2}{|l|}{\bf Algorithm} \\  
     \hline     
   $\tau$ & minibatch size  \\ 
   $\gamma$ & step size \\    
   $S, S^k$ & a random subset of $\{1,2,\dots,n\}$ \\
   $x^k$ & $k$th iterate \\
   $x^a$ & a vector chosen uniformly at random from the set of iterates $\{x^0,x^1,\dots,x^k\}$ \\
   $\phi_i^{k}$ & auxiliary variables maintained by the algorithm, $i=1,2,\dots,n$ \\
   $\bar \phi^{k}$ & $\frac{1}{n}\sum \limits_{i=1}^n \phi_i^{k}$ \\
     \hline
   \multicolumn{2}{|l|}{\bf Analysis} \\  
     \hline        
      ${\cal W}^k$ & $ \sum \limits_{i=1}^n\|\phi_i^{k} - x^* - \gamma f^{\prime}_i(\phi_i^{k}) + \gamma f^{\prime}_i(x^*)\|^2$ \\
   $\Psi^k_p$ & Lyapunov function (convex case) $\|x^{k} - x^*\|^2 + \frac{\cA \tau(2+p)}{n^3}{\cal W}^{k}$  \\
   $\Psi^k$& Lyapunov function (nonconvex case)  $ f(x^k) + \alpha \cdot \frac{1}{n} \sum \limits_{i=1}^n \|x^k - \phi^k_i\|^2$ \\
     \hline
  
 \end{longtable}

\section{Basic Facts and Simple Results}

\begin{lemma}
	If $f_i$ is convex and $L$-smooth and $f$ is convex and $L_f$-smooth, then $\forall x\in \R^d$, 
	\begin{equation}\label{eq:sumfi}
	\sum_{i=1}^n\|f^{\prime}_i(x) - f^{\prime}_i(x^*) \|^2 \leq 2nL (f(x) - f(x^*) ), 
	\end{equation}
	and 
	\begin{equation}\label{eq:Lf}
	\|f^{\prime}(x) \|^2 \leq 2L_f (f(x) - f(x^*)). 
	\end{equation}
\end{lemma}

\begin{proof}
	
	Since $f$ is convex and $L_f$ smooth, we have 
	$$
	\|f^{\prime}(x) -f^{\prime}(y)\|^2 \leq 2L_f(f(x) - f(y) - \langle f^{\prime}(y), x - y \rangle). 
	$$
	By choosing $y=x^*$, and noticing $f^{\prime}(x^*) =0$, we can get (\ref{eq:Lf}). 
	
	\vskip 2mm 
	
	Since $f_i$ is convex and $L$-smooth, we have 
	$$
	\|f^{\prime}_i(x) - f^{\prime}_i(y)\|^2 \leq 2L(f_i(x) - f_i(y) - \langle  f^{\prime}_i(y), x-y \rangle )
	$$
	which implies that 
	$$
	\sum_{i=1}^n \|f^{\prime}_i(x) - f^{\prime}_i(y)\|^2 \leq 2nL (f(x) - f(y) - \langle f^{\prime}(y), x-y \rangle )  
	$$
	By choosing $y=x^*$, and noticing $\nabla f(x^*) =0$, we can get (\ref{eq:sumfi}). 
	
\end{proof}

\section{Proofs: Section~\ref{sec:mini} }

\subsection{Proof of Lemma~\ref{lm:CDexist}}

We will show that for all vectors $a_1$, ..., $a_n \in \R^d$, we have 
$$
\mathbb{E}_{S \sim \cP} \left[  \left\| \sum_{i \in S} \frac{a_i}{p_i} \right\|^2  \right]  \leq \max_{1\leq i \leq n} \frac{\mathbb{E}^i[|S|]}{p_i} \sum_{i=1}^n \|a_i\|^2,
$$
which means that inequality \eqref{eq:sampling_asumption} holds with $\cA =\max_{1\leq i \leq n} \frac{\mathbb{E}^i[|S|]}{p_i} $ and $\cB =0$. Indeed, by the convexity of $\|\cdot\|^2$, we have 
\begin{eqnarray*}
\mathbb{E}_{S \sim \cP} \left[  \left\| \sum_{i \in S} \frac{a_i}{p_i} \right\|^2  \right]  &\leq& \mathbb{E}_{S \sim \cP} \left[ |S| \sum_{i\in S} \left\|\frac{a_i}{p_i} \right\|^2  \right] \\
&=& \sum_{C} p_C |C| \sum_{i\in C} \frac{1}{p_i^2} \|a_i\|^2 \\
&=& \sum_{i=1}^n \frac{1}{p_i} \sum_{C: i\in C} \frac{p_C}{p_i}|C| \|a_i\|^2 \\ 
&=& \sum_{i=1}^n \frac{\mathbb{E}^i[|S|]}{p_i} \|a_i\|^2 \\
&\leq& \max_{1\leq i \leq n} \frac{\mathbb{E}^i[|S|]}{p_i} \sum_{i=1}^n \|a_i\|^2. 
\end{eqnarray*}

\subsection{Proof of Lemma~\ref{lm:tauniceCD}}

	Since ${\bf P}_{ij} = \frac{\tau(\tau-1)}{n(n-1)}$ for $\tau$-nice sampling, we have 
	\begin{eqnarray*}
		\mathbb{E}\left[ \left\|\sum_{i\in S} \frac{a_i}{p_i}\right\|^2 \right] &=& \frac{n^2}{\tau^2}\mathbb{E}\left[ \sum_{i, j\in S} \langle a_i, a_j\rangle \right] \\
		&=& \frac{n^2}{\tau^2}\sum_{C} p_C \sum_{i, j\in S} \langle a_i, a_j\rangle \\ 
		&=& \frac{n^2}{\tau^2} \sum_{i, j=1}^n \sum_{C: i, j \in C}p_C \langle a_i, a_j\rangle \\
		& = & \frac{n^2}{\tau^2} \sum_{i, j=1}^n {\bf P}_{ij} \langle a_i, a_j\rangle \\ 
		&=& \frac{n(\tau-1)}{\tau(n-1)} \sum_{i\neq j} \langle a_i, a_j\rangle + \frac{n}{\tau}\sum_{i=1}^n\|a_i\|^2 \\ 
		&=&  \frac{n(\tau-1)}{\tau(n-1)} \sum_{i, j=1}^n \langle a_i, a_j\rangle + \frac{n(n-\tau)}{\tau(n-1)} \sum_{i=1}^n\|a_i\|^2 \\ 
		&=& \frac{n(n-\tau)}{\tau(n-1)} \sum_{i=1}^n\|a_i\|^2 + \frac{n(\tau-1)}{\tau(n-1)} \left\|\sum_{i=1}^na_i\right\|^2. 
	\end{eqnarray*}

\clearpage
\section{Proofs: Section~\ref{sec:ALG}}

\subsection{Proof of Lemma~\ref{lm:???}}

For $U\subseteq [n] = \{1,2,\dots,n\}$, and $S\sim \cP$, let $p_U \eqdef \mathbb{P}[S = U]$. 	Since 
	\begin{equation}\label{eq:xk+1-k}
	x^{k+1} - x^k = \frac{1}{n}\sum_{i\in S^k}(x^k - \phi_i^k - \gamma f^{\prime}_i (x^k) + \gamma f^{\prime}_i(\phi_i^k) ),
	\end{equation}
	we have 
	\begin{eqnarray*}
		\mathbb{E}_k[x^{k+1} - x^k] &=& \frac{1}{n}\sum_{U\subseteq [n]} p_U \sum_{i\in U}(x^k - \phi_i^k - \gamma f^{\prime}_i (x^k) + \gamma f^{\prime}_i(\phi_i^k) )\\ 
		&=& \frac{1}{n} \sum_{i=1}^n \sum_{U: i\in U} p_U (x^k - \phi_i^k - \gamma f^{\prime}_i (x^k) + \gamma f^{\prime}_i(\phi_i^k) ) \\ 
		&=& \frac{1}{n} \sum_{i=1}^n p_i(x^k - \phi_i^k - \gamma f^{\prime}_i (x^k) + \gamma f^{\prime}_i(\phi_i^k) ) \\ 
		&=& \frac{1}{n}\left(\tau x^k - \tau {\bar \phi}^k - \gamma \tau f^{\prime}(x^k) + \gamma \tau \frac{1}{n}\sum_if^{\prime}_i(\phi^k_i)\right) \\
		&=& -\frac{\gamma\tau}{n} f^{\prime}(x^k) \;.
	\end{eqnarray*}

\section{Proofs: Section~\ref{sec:convex_and_str_convex}}

\subsection{Proof of Lemma~\ref{lm:Dk+1}}

	\begin{eqnarray*}
		\mathbb{E}_k[{\cal W}^{k+1}] 
		&=& \sum_{i=1}^n \mathbb{E}_k[\|\phi_i^{k+1} - x^* - \gamma f^{\prime}_i(\phi_i^{k+1}) + \gamma f^{\prime}_i(x^*)\|^2] \\ 
		&=& \sum_{i=1}^n\left(1-\frac{\tau}{n}\right)\|\phi_i^{k} - x^* - \gamma f^{\prime}_i(\phi_i^{k}) + \gamma f^{\prime}_i(x^*)\|^2 + \sum_{i=1}^n \frac{\tau}{n} \|x^k -x^* - \gamma f^{\prime}_i(x^k) + \gamma f^{\prime}_i(x^*) \|^2 \\ 
		&=& \left(1-\frac{\tau}{n}\right){\cal W}^k + \frac{\tau}{n} \left(  n\|x-x^*\|^2 + \gamma^2 \sum_{i=1}^n \|f^{\prime}_i(x^k) - f^{\prime}_i(x^*) \|^2 - 2n\gamma \langle x^k -x^*, f^{\prime}(x^k) - f^{\prime}(x^*) \rangle \right) \\ 
		&\leq& \left(1-\frac{\tau}{n}\right){\cal W}^k + \tau\|x^k - x^*\|^2 + \frac{\tau\gamma^2}{n}\sum_{i=1}^n \|f^{\prime}_i(x^k) - f^{\prime}_i(x^*) \|^2\\ 
		&\overset{(\ref{eq:sumfi})}{\leq}& \left(1-\frac{\tau}{n}\right){\cal W}^k + \tau\|x^k - x^*\|^2 + 2L\tau\gamma^2(f(x^k) - f(x^*)) \;. 
	\end{eqnarray*}

\subsection{Proof of Lemma~\ref{Lem:key}}

	\begin{eqnarray*}
		\mathbb{E}_k[\|x^{k+1} - x^*\|^2] &=& \|x^k-x^*\|^2 + 2\mathbb{E}_k[\langle x^{k+1}- x^k, x^k-x^* \rangle] + \mathbb{E}_k[\|x^{k+1} - x^k\|^2] \\ 
		&\overset{(\ref{eq:xiugd97g9f})}{=}&  \|x^k-x^*\|^2 - \frac{2\gamma\tau}{n}\langle f^{\prime}(x^k), x^k-x^* \rangle +  \mathbb{E}_k[\|x^{k+1} - x^k\|^2] \\ 
		&\leq& \left(1-\frac{\tau \gamma \mu}{n}\right)\|x^k -x^*\|^2 - \frac{2\gamma\tau}{n}(f(x^k) - f(x^*)) + \mathbb{E}_k[\|x^{k+1} - x^k\|^2] \;.
	\end{eqnarray*}
	
	From Assumption \ref{as:S} and the fact that $p_i= \frac{\tau}{n}$, we have 
	\begin{eqnarray*}
		\mathbb{E}_k[\|x^{k+1} - x^k\|^2] &\overset{(\ref{eq:xk+1-k})}{=}& \frac{1}{n^2} \mathbb{E}_k \left[ \frac{\tau^2}{n^2}\left\| \sum_{i\in S^k}\frac{1}{p_i}(x^k - \phi_i^k - \gamma f^{\prime}_i (x^k) + \gamma f^{\prime}_i(\phi_i^k) ) \right\|^2 \right] \\ 
		&\leq& \frac{\tau^2\cA }{n^4} \sum_{i=1}^n \|x^k - \phi_i^k - \gamma f^{\prime}_i (x^k) + \gamma f^{\prime}_i(\phi_i^k) \|^2 + \frac{\tau^2 \cB}{n^4} \left\| \sum_{i=1}^n (x^k - \phi_i^k - \gamma f^{\prime}_i (x^k) + \gamma f^{\prime}_i(\phi_i^k) ) \right\|^2 \\ 
		&=& \frac{\tau^2 \cA}{n^4} \sum_{i=1}^n \|x^k - \phi_i^k - \gamma f^{\prime}_i (x^k) + \gamma f^{\prime}_i(\phi_i^k) \|^2 + \frac{\tau^2 \cB}{n^2}\gamma^2 \|f^{\prime}(x^k)\|^2 \\ 
		&\overset{(\ref{eq:Lf})}{\leq}& \frac{\tau^2 \cA}{n^4} \sum_{i=1}^n \|x^k - \phi_i^k - \gamma f^{\prime}_i (x^k) + \gamma f^{\prime}_i(\phi_i^k) \|^2 + \frac{2\tau^2 \cB L_f}{n^2}\gamma^2 (f(x^k) - f(x^*)). 
	\end{eqnarray*}
	
	For the first term in the above last inequality, we have 
	\begin{eqnarray*}
		&&\sum_{i=1}^n \|x^k - \phi_i^k - \gamma f^{\prime}_i (x^k) + \gamma f^{\prime}_i(\phi_i^k) \|^2 \\
		&\leq& 2\sum_{i=1}^n \left(  \|x^k - x^* - (\phi^k_i - x^* - \gamma f^{\prime}_i(\phi^k_i) + \gamma f^{\prime}_i(x^*))\|^2+ \gamma^2\| f^{\prime}_i(x^k) - f^{\prime}_i(x^*)\|^2 \right) \\ 
		&=& 2\sum_{i=1}^n \|\phi^k_i - x^* - \gamma f^{\prime}_i(\phi^k_i) + \gamma f^{\prime}_i(x^*)\|^2 - 2n\|x^k - x^*\|^2 + 2\gamma^2 \sum_{i=1}^n \| f^{\prime}_i(x^k) - f^{\prime}_i(x^*)\|^2 \\ 
		&=& 2{\cal W}^k - 2n\|x^k - x^*\|^2 + 2\gamma^2 \sum_{i=1}^n \| f^{\prime}_i(x^k) - f^{\prime}_i(x^*)\|^2 \\ 
		&\overset{(\ref{eq:sumfi})}{\leq}&  2{\cal W}^k - 2n\|x^k - x^*\|^2 + 4nL\gamma^2 (f(x^k) - f(x^*)). 
	\end{eqnarray*}
	
	Combining all the above results and Lemma \ref{lm:Dk+1}, we can obtain 
	\begin{eqnarray*}
		&& \mathbb{E}_k\left[ \|x^{k+1} - x^*\|^2 + \frac{\cA \tau(2+p)}{n^3}{\cal W}^{k+1}  \right] \\ 
		&\leq& \left(1 - \frac{\tau \gamma \mu}{n} - \frac{2\cA\tau^2}{n^3} + \frac{\cA\tau^2(2+p)}{n^3} \right) \|x^k -x^*\|^2 \\ 
		&& - \left(  \frac{2\gamma\tau}{n} - \frac{2\cB L_f\tau^2\gamma^2}{n^2} - \frac{4\cA L\tau^2\gamma^2}{n^3} - \frac{2(2+p)\cA L\tau^2\gamma^2}{n^3}  \right)(f(x^k) - f(x^*)) \\
		&& + \left(  \frac{2\cA \tau^2}{n^4} + \left(1-\frac{\tau}{n} \right)\frac{\cA \tau(2+p)}{n^3}  \right){\cal W}^k \\ 
		&=&\left(  1 - \frac{\tau}{n} \left(\gamma \mu - \frac{\cA \tau p}{n^2} \right) \right)
		\|x^k -x^*\|^2 + \frac{\cA\tau (2+p)}{n^3} \left(  1 - \frac{\tau}{n}\cdot \frac{p}{2+p}  \right) {\cal W}^k  \\ 
		&& - \frac{2\gamma\tau}{n}\left( 1 - \frac{\tau\gamma}{n}\left(\cB L_f + \frac{(4+p)\cA L}{n} \right)   \right) (f(x^k) - f(x^*)). 
	\end{eqnarray*}

\subsection{Proof of Theorem~\ref{Th:stronglyconvex}}

	If $0\leq p \leq 2$, then $\gamma = \frac{n}{\tau {\cal L}}$ implies that 
	$$
	\frac{\tau\gamma}{n}\left(\cB L_f + \frac{(4+p)\cA L}{n}\right) = \frac{\cB L_f + (4+p)\cA L/n}{{\cal L}} \leq 1 \;.
	$$
	Hence, by Lemma~\ref{Lem:key}, we have 
	\begin{equation}\label{eq:Lya}
	\mathbb{E}_k\left[ \Psi^{k+1}_p  \right] \leq \left(  1 - \frac{\tau}{n} \left(\gamma \mu - \frac{\cA\tau p}{n^2} \right) \right)\|x^k -x^*\|^2 + \frac{\cA \tau (2+p)}{n^3} \left(  1 - \frac{\tau}{n}\cdot \frac{p}{2+p}  \right) {\cal W}^k \;.
	\end{equation}
	
	We discuss two cases. 
	
	\vskip 2mm
	
	{\bf Case 1.} Suppose $n^3 \geq \frac{4\cA \tau^2{\cal L}}{\mu}$. In this case, we choose $p=2$. First, we have 
	$$
	\frac{2\cA \tau/n^2}{\gamma \mu /2} = \frac{2\cA \tau}{n^2}\cdot \frac{2}{\gamma \mu} = \frac{4\cA \tau^2{\cal L}}{n^3 \mu} \leq 1, 
	$$
	which indicates 
	$$
	\frac{\tau}{n} (\gamma \mu - \frac{\cA \tau p}{n^2}) \geq \frac{\tau \gamma \mu}{2n} = \frac{\mu}{2{\cal L}}. 
	$$
	Furthermore, 
	$$
	\frac{\tau}{n}\cdot \frac{p}{2+p} = \frac{\tau}{2n}. 
	$$
	Therefore, by (\ref{eq:Lya}), we have 
	$$
	\mathbb{E}_k[\Psi^{k+1}_p] \leq \left( 1 - \min\left\{ \frac{\tau}{2n}, \frac{\mu}{2{\cal L}}  \right\}  \right) \Psi^k_p. 
	$$
	
	Noticing $p=2$ in this case, we have 
	$$
	\mathbb{E}[\|x^k -x^*\|^2] \leq \mathbb{E}[\Psi^k_p] \leq \left( 1 - \min\left\{ \frac{\tau}{2n}, \frac{\mu}{2{\cal L}}  \right\}  \right)^k \left( \|x^0-x^*\|^2 + \frac{4\cA \tau}{n^3}{\cal W}^0 \right). 
	$$
	
	{\bf Case 2.} Suppose $n^3 < \frac{4\cA \tau^2{\cal L}}{\mu}$. In this case, we choose $p = \frac{\gamma \mu n^2}{2\cA \tau}$. First, we have 
	$$
	p = \frac{\gamma \mu n^2}{2\cA \tau} = \frac{n^3\mu}{2\cA \tau^2{\cal L}} < 2, 
	$$
	and 
	$$
	\frac{\tau}{n} (\gamma \mu - \frac{\cA \tau p}{n^2}) = \frac{\tau \gamma \mu}{2n} = \frac{\mu}{2{\cal L}}. 
	$$
	Moreover, 
	$$
	\frac{\tau}{n}\cdot \frac{p}{2+p} = \frac{\tau}{n} \cdot \frac{n^3\mu}{n^3\mu + 4\cA \tau^2{\cal L}} \geq \frac{\tau \mu n^2}{8\cA \tau^2{\cal L}} = \frac{\mu n^2}{8\cA \tau {\cal L}}. 
	$$
	Hence, by (\ref{eq:Lya}), we have 
	$$
	\mathbb{E}_k[\Psi^{k+1}_p] \leq \left( 1 - \min\left\{ \frac{\mu}{2{\cal L}}, \frac{\mu n^2}{8\cA \tau {\cal L}}  \right\}  \right) \Psi^k_p. 
	$$
	Noticing $p<2$ in this case, we have 
	$$
	\mathbb{E}[\|x^k -x^*\|^2] \leq \mathbb{E}[\Psi^k_p] \leq  \left( 1 - \min\left\{ \frac{\mu}{2{\cal L}}, \frac{\mu n^2}{8\cA \tau {\cal L}}  \right\}  \right)^k \left( \|x^0-x^*\|^2 + \frac{4\cA \tau}{n^3}{\cal W}^0 \right). 
	$$

\subsection{Proof of Corollary~\ref{co:stronglyconvexns}}

From Lemma \ref{lm:tauniceCD}, we can choose $\cA = \frac{n(n-\tau)}{\tau(n-1)}$ and $\cB =1$ for $\tau$-nice sampling. Then for $n\geq 4$, we have 
$$
\frac{\mu n^2}{8\cA \tau {\cal L}} = \frac{\mu n}{8{\cal L}}\cdot \frac{n-1}{n-\tau} \geq \frac{\mu}{2{\cal L}}, 
$$
where ${\cal L} = L_f + \frac{6L}{\tau}\cdot \frac{n-\tau}{n-1}$. Hence, from Theorem \ref{Th:stronglyconvex}, by choosing $\gamma = \frac{n}{\tau L_f + 6L(n-\tau)/(n-1)}$, we have 
$$
\mathbb{E}[\|x^k - x^*\|^2] \leq \left( 1 - \min\left\{ \frac{\tau}{2n}, \frac{\mu}{2{\cal L}}  \right\}  \right)^k \left( \|x^0-x^*\|^2 + \frac{4\cA \tau}{n^3}{\cal W}^0 \right). 
$$

\subsection{Proof of Theorem~\ref{Th:gconvex}}

From Lemma~\ref{Lem:key}, by choosing $\mu = 0$ and $p=0$, we have 
$$
\mathbb{E}_k[\Psi^{k+1}_0] \leq \Psi^k_0  - \frac{2\gamma\tau}{n}\left( 1 - \frac{\tau\gamma}{n}\left(\cB L_f + \frac{4\cA L}{n} \right)   \right) (f(x^k) - f(x^*)). 
$$
Taking expectations again and applying the tower property, we have 
$$
 \frac{2\gamma\tau}{n}\left( 1 - \frac{\tau\gamma}{n}\left(\cB L_f + \frac{4\cA L}{n} \right)   \right) \mathbb{E}[f(x^k) - f(x^*)] \leq \mathbb{E}[\Psi^k_0] - \mathbb{E}[\Psi^{k+1}_0].  
$$
Since $\gamma = \frac{n}{2\tau(\cB L_f + 4\cA L/n)}$, we have 
$$
 \mathbb{E}[f(x^k) - f(x^*)] \leq \frac{n}{\gamma \tau} \mathbb{E}[\Psi^k_0] - \mathbb{E}[\Psi^{k+1}_0], 
$$
which implies that 
\begin{eqnarray*}
\mathbb{E}[f(x^a) - f(x^*)] &=& \frac{1}{k+1} \sum_{i=0}^k \mathbb{E}[ f(x^i) - f(x^*) ] \\ 
&\leq& \frac{1}{k+1} \cdot \frac{n}{\gamma \tau} \left(  \Psi^0 - \mathbb{E}[\Psi^{k+1}_0]  \right) \\ 
&\leq&  \frac{1}{k+1} \cdot \frac{n}{\gamma \tau} \cdot \Psi^0_0 \\
&=&2\left(\cB L_f + \frac{4\cA L}{n}\right) \frac{ \left(\|x^0 - x^*\|^2 + \frac{2\cA \tau}{n^3}{\cal W}^0 \right)}{k+1}. 
\end{eqnarray*}

\section{Proofs: Section~\ref{sec:non-convex}}

\subsection{Proof of Lemma~\ref{lm:dxnonconvex}}

From (\ref{eq:xk+1-k}) and $p_i = \frac{\tau}{n}$, we have 
\begin{eqnarray*}
\mathbb{E}_k[\|x^{k+1} - x^k\|^2] &=& \frac{1}{n^2}\mathbb{E}_k\left[ \left\| \sum_{i\in S^k} \left(  x^k - \phi_i^k - \gamma f^{\prime}_i(x^k) + \gamma f^{\prime}_i(\phi^k_i)   \right) \right\|^2 \right] \\
&=& \frac{\tau^2}{n^4}\mathbb{E}_k\left[ \left\| \sum_{i\in S^k} \frac{1}{p_i} \left(  x^k - \phi_i^k - \gamma f^{\prime}_i(x^k) + \gamma f^{\prime}_i(\phi^k_i)   \right) \right\|^2 \right] \\ 
&\overset{\eqref{eq:sampling_asumption}}{\leq}& \frac{\tau^2\cA }{n^4}\sum_{i=1}^n\| x^k - \phi_i^k - \gamma f^{\prime}_i(x^k) + \gamma f^{\prime}_i(\phi^k_i) \|^2 + \frac{\tau^2\gamma^2\cB}{n^2} \|f^{\prime}(x^k)\|^2. 
\end{eqnarray*}

For the first term, since $f_i$ is $L$-smooth, we have 
\begin{eqnarray*}
\| x^k - \phi_i^k - \gamma f^{\prime}_i(x^k) + \gamma f^{\prime}_i(\phi^k_i) \|^2 &\leq& 2\|x^k - \phi^k_i\|^2 + 2\gamma^2 \|f^{\prime}_i(x^k) - f^{\prime}_i(\phi^k_i)\|^2 \\ 
&\leq& 2(1 + \gamma^2L^2)\|x^k - \phi^k_i\|^2. 
\end{eqnarray*}
Thus, 
$$
\mathbb{E}_k[\|x^{k+1} - x^k\|^2] \leq \frac{2\tau^2\cA(1+\gamma^2L^2)}{n^3} \cdot \frac{1}{n}\sum_{i=1}^n \|x^k - \phi^k_i \|^2 + \frac{\tau^2\gamma^2\cB}{n^2}\|f^{\prime}(x^k)\|^2. 
$$

\subsection{Proof of Lemma~\ref{lm:phik+1nonconvex}}

First, we have 
\begin{eqnarray*}
\frac{1}{n} \sum_{i=1}^n\|x^{k+1} - \phi^{k+1}_i \|^2 &=& \frac{1}{n} \sum_{i=1}^n \|x^{k+1} - x^k + x^k - \phi^{k+1}_i\|^2 \\
&=& \|x^{k+1} - x^k\|^2 + \frac{1}{n} \sum_{i=1}^n \|x^k - \phi^{k+1}_i \|^2 + \frac{2}{n} \sum_{i=1}^n \langle x^{k+1} -x^k, x^k - \phi^{k+1}_i \rangle. 
\end{eqnarray*}

For the third term, we have 
$$
\frac{2}{n} \sum_{i=1}^n \langle x^{k+1} -x^k, x^k - \phi^{k+1}_i \rangle = \frac{2}{n} \sum_{i=1}^n \langle x^{k+1} -x^k, x^k-\phi^k_i \rangle - \frac{2}{n} \sum_{i\in S^k} \langle x^{k+1}-x^k, x^k - \phi^k_i \rangle. 
$$

If $|S^k| = 0$, then $\sum_{i\in S^k} \langle x^{k+1}-x^k, x^k - \phi^k_i \rangle = 0$. If $|S^k| \geq 1$, then 
\begin{eqnarray*}
\left| \frac{2}{n} \sum_{i\in S^k} \langle x^{k+1}-x^k, x^k - \phi^k_i \rangle   \right| &\leq& \frac{1}{n} \sum_{i\in S^k} \left|  2 \langle x^{k+1}-x^k, x^k - \phi^k_i \rangle \right| \\ 
&\leq& \frac{1}{n} \sum_{i\in S^k} \left(  \frac{6Mn}{\tau|S^k|} \|x^{k+1} -x^k\|^2 + \frac{\tau|S^k|}{6Mn} \|x^k - \phi^k_i \|^2  \right) \\ 
&=& \frac{6M}{\tau}\|x^{k+1} -x^k\|^2 + \frac{\tau}{6Mn^2} \sum_{i\in S^k} |S^k| \cdot \|x^k - \phi^k_i \|^2. 
\end{eqnarray*}

Thus, 
\begin{eqnarray*}
\frac{1}{n} \sum_{i=1}^n\|x^{k+1} - \phi^{k+1}_i \|^2  &\leq& \left(  \frac{6M}{\tau} +1  \right) \|x^{k+1} - x^k\|^2 +  \frac{1}{n} \sum_{i=1}^n \|x^k - \phi^{k+1}_i \|^2 \\ 
&& + \frac{2}{n} \sum_{i=1}^n \langle x^{k+1} -x^k, x^k-\phi^k_i \rangle + \frac{\tau}{6Mn^2} \sum_{i\in S^k} |S^k| \cdot \|x^k - \phi^k_i \|^2.  
\end{eqnarray*}

For the second term in above inequality, we have 
$$
\mathbb{E}_k\left[\frac{1}{n} \sum_{i=1}^n \|x^k - \phi^{k+1}_i \|^2 \right] = \left(1-\frac{\tau}{n} \right)\frac{1}{n} \sum_{i=1}^n \|x^k - \phi^k_i \|^2. 
$$
For the third term, we have 
\begin{eqnarray*}
\mathbb{E}_k\left[ \frac{2}{n} \sum_{i=1}^n \langle x^{k+1} -x^k, x^k-\phi^k_i \rangle \right] &=& \frac{2}{n} \sum_{i=1}^n \langle -\frac{\gamma \tau}{n}f^{\prime}(x^k), x^k - \phi^k_i \rangle \\ 
&=& -\frac{2\gamma \tau}{n} \langle f^{\prime}(x^k), x^k - {\bar \phi}^k \rangle \\ 
&\leq& \frac{\gamma\tau}{n} \left(  \frac{1}{\beta} \|f^{\prime}(x^k)\|^2 + \beta \|x^k - {\bar \phi}^k\|^2  \right) \\ 
&=& \frac{\gamma\tau}{n} \left(  \frac{1}{\beta} \|f^{\prime}(x^k)\|^2 + \beta \frac{1}{n^2}\left\|  \sum_{i=1}^n (x^k - \phi^k_i)  \right\|^2  \right) \\ 
&\leq&  \frac{\gamma\tau}{n} \left(  \frac{1}{\beta} \|f^{\prime}(x^k)\|^2 + \beta \cdot \frac{1}{n} \sum_{i=1}^n \|x^k - \phi^k_i \|^2  \right), 
\end{eqnarray*}
where the first inequality is from Young's inequality, and the second inequality is from the convexity of the norm $\|\cdot \|^2$. \\ 

For the fourth term, we have 
\begin{eqnarray*} 
\mathbb{E}_k\left[\frac{\tau}{6Mn^2} \sum_{i\in S^k} |S^k| \cdot \|x^k - \phi^k_i \|^2 \right] &=& \frac{\tau}{6Mn^2} \sum_{U\subseteq [n]}\sum_{i\in U} p_U |U| \cdot \|x^k - \phi^k_i\|^2 \\ 
&=& \frac{\tau}{6Mn^2} \sum_{i=1}^n p_i \sum_{U: i\in U} \frac{p_U}{p_i}|U| \cdot \|x^k - \phi^k_i\|^2 \\ 
&=& \frac{\tau^2}{6Mn^3} \sum_{i=1}^n \mathbb{E}^i[|S|] \cdot \|x^k -\phi^k_i \|^2 \\ 
&\leq& \frac{\tau^2}{6n^2}\cdot \frac{1}{n} \sum_{i=1}^n \|x^k -\phi^k_i \|^2, 
\end{eqnarray*}
where $p_U \eqdef \mathbb{P}[S = U]$ and $\mathbb{E}^i[|S|] =  \sum_{U: i\in U} \frac{p_U}{p_i}|U|$. \\

Combining all the above results, we obtain the result.

\subsection{Proof of Theorem~\ref{Th:nonconvex1}}

Since $f$ is $L_f$-smooth, we have 
$$
f(x^{k+1}) \leq f(x^k) + \langle \nabla f(x^k), x^{k+1}-x^k \rangle + \frac{L_f}{2}\|x^{k+1}-x^k\|^2, 
$$
which implies 
$$
\mathbb{E}_k[f(x^{k+1})] \leq f(x^k) - \frac{\gamma \tau}{n}\|f^{\prime}(x^k)\|^2 + \frac{L_f}{2} \mathbb{E}_k[\|x^{k+1} - x^k\|^2]. 
$$

Hence, we have 
\begin{eqnarray*}
\mathbb{E}_k[\Psi^{k+1}] &=& \mathbb{E}_k [f(x^{k+1}) + \alpha \cdot \frac{1}{n} \sum_{i=1}^n \|x^{k+1} - \phi^{k+1}_i\|^2] \\
&\leq&  f(x^k) - \frac{\gamma \tau}{n}\|f^{\prime}(x^k)\|^2 + \frac{L_f}{2} \mathbb{E}_k[\|x^{k+1} - x^k\|^2] \\ 
&& + \alpha \mathbb{E}_k\left[\frac{1}{n} \sum_{i=1}^n \|x^{k+1} - \phi^{k+1}_i\|^2 \right] \\
&\overset{Lemma ~\ref{lm:phik+1nonconvex}}{\leq}&  f(x^k) - \left(  \frac{\gamma \tau}{n} - \frac{\alpha \gamma \tau}{n\beta}  \right) \|f^{\prime}(x^k)\|^2 + \left(  \frac{L_f}{2} + \alpha \left(\frac{6M}{\tau} +1 \right)  \right) \mathbb{E}_k[\|x^{k+1} - x^k\|^2] \\ 
&& + \alpha \left( 1 - \frac{\tau}{n} + \frac{\tau^2}{6n^2} + \frac{\gamma \tau \beta}{n}  \right) \frac{1}{n} \sum_{i=1}^n \|x^k - \phi^k_i\|^2  \\ 
&\overset{Lemma ~\ref{lm:dxnonconvex}}{\leq}& f(x^k) - \left(  \frac{\gamma \tau}{n} - \frac{\alpha \gamma \tau}{n\beta} - \frac{\tau^2\gamma^2\cB}{n^2}\left(  \frac{L_f}{2} + \alpha \left(\frac{6M}{\tau} +1 \right)  \right)   \right) \|f^{\prime}(x^k)\|^2 \\ 
&& +  \alpha \left( 1 - \frac{\tau}{n} + \frac{\tau^2}{6n^2} + \frac{\gamma \tau \beta}{n}  + \left(  \frac{L_f}{2\alpha} + \frac{6M}{\tau} +1  \right) \frac{2\tau^2\cA \left(1+\gamma^2L^2 \right)}{n^3}   \right) \frac{1}{n} \sum_{i=1}^n \|x^k - \phi^k_i\|^2. 
\end{eqnarray*}

Since $\alpha = \frac{\beta}{q} = \frac{1}{2q\gamma}$, we have 
$$
\frac{\alpha \gamma \tau}{n\beta} = \frac{\gamma \tau}{nq} \leq \frac{\gamma \tau}{4n} \quad {\rm and} \quad \frac{\tau^2 \gamma^2 \cB}{n^2} \cdot \alpha \left(\frac{6M}{\tau} + 1 \right) = \frac{\gamma \tau}{n} \cdot \frac{\tau}{n} \cdot \frac{\cB (\frac{6M}{\tau} + 1)}{2q} \leq \frac{\gamma \tau}{4n}
$$
Moreover, from $\gamma \leq \frac{n}{2\cB L_f\tau}$, we have $\frac{\tau^2 \gamma^2 \cB}{n^2}\cdot \frac{L_f}{2} \leq \frac{\gamma \tau}{4n}$. Hence 
$$
 \frac{\gamma \tau}{n} - \frac{\alpha \gamma \tau}{n\beta} - \frac{\tau^2\gamma^2\cB}{n^2}\left(  \frac{L_f}{2} + \alpha \left(\frac{6M}{\tau} +1 \right)  \right) \geq \frac{\gamma \tau}{4n}. 
$$
From $\frac{n^2}{\tau \cA} \geq 24 \left(\frac{6M}{\tau} + 1 \right)$, we have $\left(\frac{6M}{\tau} + 1 \right)\frac{2\tau \cA}{n^2} \leq \frac{1}{12}$. From 
$$
\gamma \leq \min \left\{  \frac{n^2/\tau \cA}{24qL_f},  \frac{(n^2/\tau \cA)^{\frac{1}{3}}}{(24qL_fL^2)^{\frac{1}{3}}} , \frac{(n^2/\tau \cA)^{\frac{1}{2}}}{(24(6M/\tau +1)L^2)^{\frac{1}{2}}}  \right\}, 
$$
we have 
$$
\frac{L_f}{2\alpha} \cdot \frac{2\tau \cA}{n^2} = L_fqr\cdot \frac{2\tau \cA}{n^2} \leq \frac{1}{12}, \quad \left(\frac{6M}{\tau} + 1 \right)\frac{2\tau \cA}{n^2}\gamma^2L^2 \leq \frac{1}{12}, 
$$
and 
$$
\frac{L_f}{2\alpha} \frac{2\tau \cA\gamma^2L^2}{n^2} = L_fq\gamma^2L^2 \cdot \frac{2\tau \cA}{n^2} \leq \frac{1}{12}. 
$$
Hence, we have 
\begin{eqnarray*}
1 - \frac{\tau}{n} + \frac{\tau^2}{6n^2} + \frac{\gamma \tau \beta}{n}  + \left(  \frac{L_f}{2\alpha} + \frac{6M}{\tau} +1  \right) \frac{2\tau^2\cA (1+\gamma^2L^2)}{n^3} &\leq& 1 - \frac{\tau}{n} + \frac{\tau^2}{6n^2} + \frac{\gamma\tau\beta}{n} + \frac{\tau}{3n} \\ 
&=& 1 - \frac{\tau}{n} + \frac{\tau^2}{6n^2} + \frac{\tau}{2n} + \frac{\tau}{3n}\\ 
&\leq& 1. 
\end{eqnarray*}
Therefore, we have 
$$
\mathbb{E}_k [\Psi^{k+1}] \leq \Psi^k - \frac{\gamma \tau}{4n}\|\nabla f(x^k) \|^2. 
$$

\subsection{Proof of Corollary~\ref{co:nonconvex1}}

If $\gamma$ satisfies (\ref{eq:gammanonconvex}), from Theorem~\ref{Th:nonconvex1}, we have 
$$
\mathbb{E}[\|\nabla f(x^k)\|] \leq \frac{4n}{\gamma \tau} (\mathbb{E}[\Psi^{k}] - \mathbb{E}[\Psi^{k+1}]), 
$$
which implies that 
\begin{eqnarray*}
	\mathbb{E}[\|\nabla f(x^a)\|^2] &=& \frac{1}{k+1} \sum_{i=0}^k \mathbb{E}[\| \nabla f(x^i)\|^2] \\ 
	&\leq& \frac{1}{k+1} \cdot \frac{4n}{\gamma \tau} \left(  \Psi^0 - \mathbb{E}[\Psi^{k+1}]  \right) \\ 
	&=&  \frac{1}{k+1} \cdot \frac{4n}{\gamma \tau} \left(  f(x^0) -  \mathbb{E}[f(x^{k+1})] - \alpha \mathbb{E}\left[\frac{1}{n}\sum_{i=1}^n \|x^{k+1} -\phi^{k+1}_i\|^2 \right]  \right) \\
	&\leq& \frac{4n}{\gamma \tau}\cdot \frac{f(x^0) - f(x^*)}{k+1}. 
\end{eqnarray*}

\clearpage

\section{Comparison of Linear Rates for Strongly Convex $f$}\label{sec:table3}

\begin{table}[h]
	\begin{center}
		\begin{tabular}{|c|c|c|c|}
			\hline
			Reference & Stepsize & Convergence rate &  Assumptions  \\
			\hline
			\begin{tabular}{c}
				\tiny Mairal \\ICML 2013~\cite{mairal2013optimization}\end{tabular} 
			& $\nicefrac{1}{L}$ & $\left(1-\tfrac{2\mu}{n(L+ \mu)}\right)^k$  &  \begin{tabular}{c}  $\mu$-strong convexity \end{tabular} \\
			\hline
			\begin{tabular}{c}
				\tiny Mairal \\ICML 2013~\cite{mairal2013optimization}\end{tabular} 
			& $\nicefrac{1}{\mu}$ &  $\frac{n}{\mu}\left(1-\tfrac{1}{3n}\right)^k$  &  \begin{tabular}{c} $\mu$-strong convexity \\ $\kappa = {\cal O}(n)$ \end{tabular} \\
			\hline
			\begin{tabular}{c}\tiny Defazio, Caetano \& Domke\\ICML 2014~\cite{defazio2014finito}\end{tabular}
			& $\nicefrac{1}{\mu}$ & $\tfrac{1}{\mu}\left(1-\tfrac{1}{2n}\right)^k$ &  \begin{tabular}{c} $\mu$-strong convexity \\ $\kappa = {\cal O}(n)$ \end{tabular}\\
			\hline
			\begin{tabular}{c}\tiny Lin, Mairal \& Harchaoui\\NIPS 2015~\cite{lin2015universal} \end{tabular} & $\nicefrac{1}{\mu}$ &  $\tfrac{1}{\tau_1} \left(  1 - \tau_1  \right)^{k+1}$& \begin{tabular}{c} $\mu$-strong convexity \\ $\tau_1 \geq \min\{  \tfrac{\mu}{4L}, \frac{1}{2n}  \}$ \end{tabular} \\ 
			\hline 
			{\bf THIS WORK} & $\nicefrac{n}{\tau {\cal L}}$ & $\left( 1 - \min\left\{ \tfrac{\tau}{2n}, \tfrac{\mu}{2{\cal L}}, \tfrac{\mu n^2}{8\cA \tau {\cal L}}  \right\}  \right)^k$ & \begin{tabular}{c} $\mu$-strong convexity \\ ${\cal L} = \cB L_f+ \tfrac{6\cA L}{n} $ \\  Assumption \ref{as:S} \end{tabular} \\ 
			\hline
		\end{tabular}
		\caption{Comparison of known and new rates established for MISO in the strongly convex case.  }
			\label{tbl:rate}
	\end{center}
\end{table}

\end{document}